\documentclass[12pt]{amsart}
\usepackage{amsfonts,amssymb,amsmath}

\usepackage{mathrsfs}
\usepackage{bbm}

\usepackage{tikz}

\usepackage{xcolor}

\newcommand {\SC} {{\mathbb C}}

\newcommand {\SK} {{\mathbb K}}
\newcommand {\SN} {{\mathbb N}}
\newcommand {\SR} {{\mathbb R}}

\newcommand {\SX} {{\mathbb X}}

\newcommand {\SZ} {{\mathbb Z}}

\let\oldphi=\phi

\renewcommand {\phi} {{\varphi}}

\newcommand {\al} {{\alpha}}

\newcommand {\dt} {{\delta}}
\newcommand {\Dt} {{\Delta}}

\newcommand {\e} {{\varepsilon}}
\newcommand {\bfe} {{\boldsymbol \e}}

\newcommand {\la} {{\lambda}}
\newcommand{\La}{{\Lambda}}

\newcommand{\ga}{{\gamma}}
\newcommand{\Ga}{{\Gamma}}

\newcommand {\phin} {{\varphi_{i_{n+1}}}}

\newcommand{\be}{{\bf e}}

\newcommand {\bp} {{\bf p}}

\newcommand {\bx} {{\bf x}}
\newcommand {\by} {{\bf y}}

\newcommand {\cB} {{{\mathcal B}}}

\newcommand {\cD} {{\mathcal D}}

\newcommand {\cH} {{\mathcal H}}
\newcommand {\cI} {{\mathcal I}}

\newcommand {\cY} {{\mathcal Y}}

\newcommand {\Ts} {\textstyle}

\def\supp{\mathop{\rm supp}}
\def\dist{\mathop{\rm dist}}
\def\sgn{\mathop{\rm sign\, }}

\def\lsgn{\mathop{\overline{\mbox{\rm sign}}}}

\def\lan#1#2{{\langle\,{#1}\,,\,{#2}\,\rangle}}

\newcommand {\Scirc} {\raise.2ex\hbox{$\scriptstyle\circ$}}

\newcommand {\mand} {{\quad\mbox{and}\quad}}

\renewcommand {\mid} {{\,\,\,\colon\,\,\,}}

\renewcommand {\mid} {{\,\,\,\colon\,\,\,}}

\newcommand{\sline}{{\smallskip

\noindent}}

\newcommand{\Ba}[1]{\begin{array}{#1}}
\newcommand{\Ea}{\end{array}}
\newcommand{\Be}{\begin{equation}}
\newcommand{\Ee}{\end{equation}}
\newcommand{\Bea}{\begin{eqnarray}}
\newcommand{\Eea}{\end{eqnarray}}
\newcommand{\Beas}{\begin{eqnarray*}}
\newcommand{\Eeas}{\end{eqnarray*}}
\newcommand{\Benu}{\begin{enumerate}}
\newcommand{\Eenu}{\end{enumerate}}
\newcommand{\Bi}{\begin{itemize}}
\newcommand{\Ei}{\end{itemize}}

\newcommand{\BR}{\begin{Remark} \em}
\newcommand{\ER}{\end{Remark}}

\newcommand{\BE}{\begin{example} \em}
\newcommand{\EE}{\end{example}}

\newcounter{remark}

\newtheorem{theorem}[equation]{T{\hskip 0pt\footnotesize\bf HEOREM}}
\newtheorem{proposition}[equation]{P{\hskip 0pt\footnotesize\bf ROPOSITION}}
\newtheorem{corollary}[equation]{C{\hskip 0pt\footnotesize\bf OROLLARY}}
\newtheorem{lemma}[equation]{L{\hskip 0pt\footnotesize\bf EMMA}}
\newtheorem{Remark}[equation]{R{\hskip 0pt\footnotesize\bf EMARK}}
\newtheorem{definition}[equation]{D{\hskip 0pt\footnotesize\bf EFINITION}}
\newtheorem{example}[equation]{E{\hskip 0pt\footnotesize\bf XAMPLE}}

\newcommand {\Proofof}[1] {\noindent{\bf P{\footnotesize\bf ROOF} of {#1}: } \ }
\newcommand {\ProofEnd} {
             \begin{flushright} \vskip -0.2in $\Box$ \end{flushright}}

\newcommand {\Au}{{\texttt{A2}}}
\newcommand {\At}{{\texttt{A3}}}
\newcommand {\D}{{\texttt{D}}}
\def\bbone{{\mathbbm 1}}

\def\bp{\bbone_{I,p}}

\def\N{{\mathbb N}}

\newcommand {\tri}[1] {\big\VERT#1\big\VERT_p}
\newcommand {\stri}[1] {\VERT#1\VERT_p}

\newcommand {\fpq} {{\mathfrak f_{p,q}}}
\newcommand {\fpt} {{\mathfrak f_{p,2}}}

\newcommand {\A} {{\mathscr{A}}}
\newcommand {\G} {{\mathscr{G}}}
\newcommand {\sD} {{\mathscr{D}}}
\newcommand {\sR} {{\mathscr{R}}}

\newcommand {\sRd} {\sR_d}
\newcommand {\bD} {\bar{\sD}}
\newcommand {\bRd} {\bar{\sR}_d}

\newcommand {\bm} {{\vec{m}}}

\DeclareFontEncoding{FMS}{}{}
\DeclareFontSubstitution{FMS}{futm}{m}{n}
\DeclareFontEncoding{FMX}{}{}
\DeclareFontSubstitution{FMX}{futm}{m}{n}
\DeclareSymbolFont{fouriersymbols}{FMS}{futm}{m}{n}
\DeclareSymbolFont{fourierlargesymbols}{FMX}{futm}{m}{n}
\DeclareMathDelimiter{\VERT}{\mathord}{fouriersymbols}{152}{fourierlargesymbols}{147}

\def \<{\langle}
\def\>{\rangle}
\def \bn{{\mathbf n}}

\def\Ave{\operatorname{Ave}}

\begin{document}

\title{Lebesgue-type inequalities in greedy approximation }

\author{S. Dilworth}
\address{S. J. Dilworth
\\
Department of Mathematics, University of South Carolina, Columbia,
SC, USA
} \email{dilworth@math.sc.edu}

\author{G. Garrig\'os}
\address{Gustavo Garrig\'os
\\
Departamento de Matem\'aticas
\\
Universidad de Murcia
\\
30100 Murcia, Spain} \email{gustavo.garrigos@um.es}

\author{E. Hern\'andez}
\address{Eugenio Hern\'andez
\\
Departamento de Matem\'aticas,
Universidad Aut\'onoma de Madrid,
28049 Madrid, Spain} \email{eugenio.hernandez@uam.es}

\author{D. Kutzarova}
\address{D. Kutzarova
\\
Department of Mathematics, University of Illinois Urbana-Champaign,
	Urbana, IL 61807, USA; Institute of Mathematics and Informatics, Bulgarian Academy of
Sciences, Sofia, Bulgaria} 
\email{denka@illinois.edu}

\author{V. Temlyakov}
\address{V. N. Temlyakov
	\\
	Department of Mathematics, University of South Carolina, Columbia,
SC, USA; Steklov Institute of Mathematics, Moscow, Russia; Lomonosov
Moscow State University, Moscow, Russia}
	\email{temlyak@math.sc.edu}



\begin{abstract}
We present new results regarding Lebesgue-type inequalities for the Weak Chebyshev Greedy Algorithm (WCGA)
in uniformly smooth Banach spaces. We improve earlier bounds in \cite{Tem14} for
dictionaries satisfying a new property introduced here. We apply these results
to derive optimal bounds 
in two natural examples of sequence spaces. In particular, optimality is obtained in the case
of the multivariate Haar system in $L^p$ with $1<p\leq 2$, under the Littlewood-Paley norm.
\end{abstract}

\maketitle

\section{Introduction}

This paper is devoted to theoretical aspects of sparse approximation. The main motivation for the study of sparse approximation is that many real world signals can be well approximated by sparse ones. In a general setting 
we are working 
in a Banach space $\SX$ with a redundant system of elements (dictionary) $\cD$. There is a solid justification of the importance of a Banach space setting in numerical analysis in general, and in sparse approximation in particular; see, for instance, \cite[Preface]{Tem11}. An element (function, signal) $f\in \SX$ is said to be $N$-sparse with respect to $\cD$ if
it has a representation $f=\sum_{j=1}^Nc_jg_j$,  where $g_j\in \cD$ and $c_j$ is a scalar, $j=1,\dots,N$. The set of all $N$-sparse elements is denoted by $\Sigma_N(\cD)$. For a given element $f$ we introduce the error of best $N$-term approximation
$$
\sigma_N(f,\cD) := \inf_{a\in\Sigma_N(\cD)} \|f-a\|.
$$
In a general setting 
one studies algorithms (approximation methods) $\A = \{A_N(\cdot,\cD)\}_{N=1}^\infty$ with respect to a given dictionary $\cD$. 
These mappings must satisfy that $A_N(f,\cD)\in\Sigma_N(\cD)$, for all $f\in\SX$;
in other words, $A_N$ provides an $N$-term approximant with respect to $\cD$. It is clear that for any $f\in\SX$ and any $N$ we have
$
\|f-A_N(f,\cD)\| \ge \sigma_N(f,\cD).
$
We are interested in such pairs $(\cD,\A)$ for which the algorithm $\A$ provides approximation close to the best $N$-term approximation. We introduce the corresponding definition (see \cite{Tem18}, p.423).
Let $\vartheta(u)$ be a  function such that 
$\vartheta(u)\ge 1$. 
\begin{definition}\label{ID1} We say that $\cD$ is a $\vartheta$-greedy dictionary with respect to $\A$ if there exists a constant $C$ such that for any $f\in \SX$ and all $N\in \N$ we have
\begin{equation}\label{I1}
\|f-A_{\vartheta(N)N}(f,\D)\| \le C\sigma_N(f,\D).
\end{equation}
\end{definition}
In the case $\vartheta(u)=C_0\geq 1$ is a constant we call $\cD$ an almost greedy dictionary with respect to $\A$, and in the case $C_0=1$ we call $\cD$ a greedy dictionary. If $\cD=\Psi$ is a basis then in the above definitions we replace dictionary by basis. Inequalities of the form (\ref{I1})  are called  \emph{Lebesgue-type inequalities}. 

 In the case that $\A=\{G_N(\cdot,\Psi)\}_{N=1}^\infty$ is the Thresholding Greedy Algorithm (TGA), the theory of greedy and almost greedy bases is well developed (see \cite{Tem11} and \cite{VTsparseb}). We remind that if $\Psi=\{\psi_k\}_{k=1}^\infty$ is a normalized basis in $\SX$, and $f=\sum_{k=1}^\infty c_k\psi_k\in\SX$, then the TGA at the $N$th iteration 
gives an approximant  
$$
G_N(f,\Psi):=\sum_{j=1}^N c_{k_j}\psi_{k_j}, \quad \mbox{where}\quad|c_{k_1}|\ge |c_{k_2}| \ge \dots.
$$
In particular, it is known  (see \cite{Tem11}, p.17) that the univariate Haar system is a greedy basis with respect to TGA in $L^p$, for all $1<p<\infty$. Also, it is known that the TGA does not work well with respect to the trigonometric system (see, for instance, \cite[Ch. 8]{Tem18}).  It was demonstrated in the paper \cite{Tem14} (see also \cite{Tem18}, Ch.8) that the \emph{Weak Chebyshev Greedy Algorithm} (WCGA), which we define momentarily, works very well for a special class of dictionaries, which includes the trigonometric system.   

In this paper we further develop recent results from \cite{Tem14} concerning Lebesgue-type inequalities 
for the  WCGA in the context of uniformly smooth Banach spaces. We mostly concentrate on the case when WCGA is applied with respect to a basis. 
 We also emphasize that, although the theory of these algorithms is typically stated for \emph{real} Banach spaces,
our results are actually valid for both, real or complex Banach spaces; in particular 
they can be applied to the standard (complex) trigonometric basis.


\

We now recall how this algorithm is defined; see \cite{Tem01} or \cite[Chapter 6.2]{Tem11}.  
Let $(\SX,\|\cdot\|)$ be a Banach space over $\SK=\SR$ or $\SC$.
Given $f\in\SX\setminus\{0\}$, let $F_f$ be an associated norming functional in $\SX^*$, that is,
\Be
\|F_f\|_{\SX^*}=1,\mand F_f(f)=\|f\|.
\label{F_f}
\Ee
Such functionals always exist by the Hahn-Banach theorem, and are unique provided the norm $\|\cdot\|$ is \emph{smooth}.
Here we shall assume the stronger property that $\|\cdot\|$ is \emph{uniformly smooth of power type}, that is, 
there exists $q>1$ and a constant $\ga>0$ such that
\Be
\rho(t)\leq \,\ga\, t^q,\quad t>0,
\Ee
where $\rho(t)$ is the associated modulus of smoothness, given by
\Be
2\rho(t)=\sup_{\|f\|=\|g\|=1} \Big(\|f+tg\|+\|f-tg\|-2\|f\|\Big),\quad t\in\SR.
\label{rho}
\Ee 

Let $\cD=\{\phi_i\}_{i\in\cI}\subset\SX$ be a fixed \emph{dictionary} in $\SX$, that is, a subset of unit vectors with dense span. 
We also fix a weakness parameter $\tau\in(0,1]$. 
The WCGA associated with $(\SX,\|\cdot\|,\cD,\tau)$ is defined as follows. 

\sline {\bf Weak Chebyshev Greedy Algorithm (WCGA).} Given $f\in \SX\setminus\{0\}$, we let $f_0:=f$ and define inductively vectors $\phi_{i_1},\ldots, \phi_{i_n}$ in $\cD$ and $f_1,\ldots, f_n\in\SX$ by the following procedure:
at step $n+1$ we pick any  $\phin\in\cD$ such that
\[
|F_{f_n}(\phin)|\geq \,\tau\,\sup_{\phi\in\cD}|F_{f_n}(\phi)|,
\] 
and let $\G_{n+1}(f)$ be any element in $[\phi_{i_1},\ldots,\phi_{i_{n+1}}]$ such that
\[
\|f-\G_{n+1}(f)\|=\dist\big(f,[\phi_{i_1},\ldots,\phi_{i_{n+1}}]\big).
\]
Then we set $f_{n+1}=f-\G_{n+1}(f)$, and iterate the process (indefinitely, or until the remainder $f_{n+1}=0$). 

It is known from \cite{Tem01} that $\|f-\G_{n}(f)\|\to 0$ as $n\to\infty$, for all $f\in\SX$, provided that $(\SX,\|\cdot\|)$ is uniformly smooth.
More recently, it has been shown in \cite{Tem14} that \emph{Lebesgue-type inequalities} hold for this algorithm.
In this paper it is convenient to formulate them as follows: we search for functions $\oldphi:\SN\to\SN$ and constants $C>0$ such that,
for all $N\geq1$ it holds
\Be
\big\|f-\G_{\oldphi(N)}(f)\|\leq \,C\,\sigma_N(f), \quad \forall\,f\in\SX.
\label{oldphi}
\Ee
Note, in particular, that \eqref{oldphi} implies exact recovery for each $f\in\Sigma_N$ after $\oldphi(N)$ steps. 
Since the dictionary $\cD$ is fixed we omit it from the notation, that is, $\Sigma_N=\Sigma_N(\cD)$
and $\sigma_N(f)=\sigma_N(f,\cD)$. 

We now state a result from \cite{Tem14} which gives bounds for such functions $\oldphi$
in terms of suitable parameters depending on $(\SX,\|\cdot\|,\cD)$; see also \cite[Section 8.7]{Tem18}.
These parameters are quantified by the following properties.

\subsection*{Property A2}
 We say that $(\SX,\cD)$ satisfies property $\Au(U)$, with parameter $U\geq1$, if
\Be
\big\|\sum_{i\in A}a_i\phi_i\big\|\,\leq U\,\big\|\sum_{i\in B} a_i\phi_i\big\|,\quad 
\label{A2}
\Ee
for all finite sets $A\subset B$, and all scalars $a_i\in\SK$.
When \eqref{A2} holds  only for sets with $|A|\leq K$, we say that $\Sigma_K$ has the property $\Au(U)$.

\subsection*{Property A3}  We say that $(\SX,\cD)$ satisfies property $\At(r, V)$, with parameters $r\in(0,1]$ and $V\geq1$, if
\Be
\sum_{i\in A}|a_i|\leq V\,|A|^r\,\big\|\sum_{i\in B} a_i\phi_i\big\|,\quad 
\label{A3}
\Ee
for all finite sets $A\subset B$, and all scalars $a_i\in\SK$.
When \eqref{A3} holds  only for sets with $|A|\leq K$, we say that $\Sigma_K$ has the property $\At(r, V)$. 

A brief discussion on the meaning of these properties and some examples, is given in \cite{Tem14}; see also \cite[section 8.7]{Tem18} or
 \S\ref{S_A2A3} below. 
The next result is a special case of Theorem 2.8 in \cite{Tem14}; see also \cite[Theorem 8.7.18]{Tem18}.

\begin{theorem}\label{th_tem14}
Let $N\geq 1$ and assume that $(\SX,\|\cdot\|,\cD)$ has the following properties
\Benu
\item[(i)] $\rho(t)\leq \ga t^q$,  for some $\ga>0$ and $q\in(1,2]$;
\item[(ii)]  
$\Sigma_N$ satisfies property $\At(r, V)$, for some $r\in(0,1]$ and $V>0$;
\item[(iii)] 
$\Sigma_N$ satisfies property $\Au(U)$, for some $U\geq1$.
\Eenu
Then, there is a universal constant $C\geq1$ such that, for every $\tau\in(0,1]$, the Lebesgue-type inequality in \eqref{oldphi} holds with
\Be
\oldphi(N)\,=\,\big\lfloor 
C_1(\tau,\ga,q)\,\big(\log(U+1)\big)\,V^{q'}\, N^{rq'}\big\rfloor,
\label{phiUV}
\Ee
where $C_1(\tau,\ga,q)=C_2(q)\,\ga^{q'/q}\,\tau^{-q'}$.
\end{theorem}

\BR
\label{R_VN}
A similar result holds with $U$ replaced by $VN$ in \eqref{phiUV}, since 
$\At(V,r)$ for $\Sigma_N$  implies $\Au(U)$ with $U=VN$; see \cite[Theorem 2.7]{Tem14}.
\ER


\

In this paper we elaborate further on these results in the following directions.
First we replace the condition (i) on $(\SX,\|\cdot\|)$, by a less demanding condition involving as well the system $\cD$.

\begin{definition}\label{D_Ds}{\rm
 We say that $(\SX,\|\cdot\|,\cD)$ has the property $\D(s, c_1)$, for some parameters $c_1>0$ and $s>1$, if 
\Be
\dist(f,[\phi])\leq \|f\|\,\big(1-c_1|F_f(\phi)|^{s}\big),
\label{distFf}
\Ee
for all $f\in\SX\setminus\{0\}$ and all $\phi\in\cD$.
}
\end{definition}
 
We shall show in Proposition \ref{P_dist} below that $\rho(t)\leq \ga t^q$ implies \eqref{distFf} 
with $s=q'$ and some $c_1=c_1(\ga, q)$, actually for all  $\|\phi\|=1$. 
However, for certain $(\SX,\cD)$ it may happen $\D(s,c_1)$ holds with a better (smaller) value of $s$.
For example, if $\SX=\ell^p$, {\small $1<p<\infty$}, and $\cD$ is the canonical basis, then $\D(s,c_1)$ holds with $s=p'$, 
while the modulus of smoothness has power type $q=\min\{p,2\}$. This, and the more general example $\SX=\ell^p(\ell^q)$, will be discussed in 
\S\ref{S_lpq} below.

Using this new concept we shall show the following improvement over Theorem \ref{th_tem14}.

\begin{theorem}\label{th_Ds}
Let $N\geq1$ and assume that $(\SX,\|\cdot\|,\cD)$ has the following properties
\Benu
\item[(i)] $(\SX,\cD)$ satisfies property $\D(s, c_1)$, for some $s>1$ and $c_1>0$;
\item[(ii)] $\Sigma_N$ satisfies property $\At(r, V)$, for some $r\in(0,1]$ and $V>0$;
\item[(iii)] $\Sigma_N$ satisfies property $\Au(U)$, for some $U\geq1$.
\Eenu
Then, there is a universal constant $C\geq1$ such that, for every $\tau\in(0,1]$,
 the Lebesgue-type inequality in \eqref{oldphi} holds with
\Be
\oldphi(N)\,= \,\big\lfloor 
C'_1(\tau,c_1,s)\,\big(\log(U+1)\big)\,V^{s}\, N^{rs}\big\rfloor,
\label{phiUVs}
\Ee
where $C'_1(\tau,c_1,s)=C'_2(s)\,c_1^{-1}\,\tau^{-s}$.
\end{theorem}

As an application, we shall find {\bf optimal} Lebesgue type inequalities for
some special pairs $(\SX,\cD)$. 
We remark that no lower bounds for functions $\oldphi$ satisfying \eqref{oldphi}
seemed to appear earlier in the literature (besides the trivial $\oldphi(N)\geq N$).
Our first example concerns the space $\SX=\ell^p(\ell^q)$ with $\cD$ the canonical basis. 

\begin{theorem}\label{th_lpq}
Let $\SX=\ell^p(\ell^q)$, with $1<p,q<\infty$, and $\cD$ be the canonical basis.
Let
\[
\beta=\beta(p,q)=\max\Big\{\frac{p'}{q'}, \frac{q'}{p'}\Big\}.
\]
Then, there exists $c=c(p,q,\tau)>0$ such that the WCGA satisfies \eqref{oldphi}
with
\[
\oldphi(N)\,=\,\big\lfloor 
c\, N^{\beta}\big\rfloor. 
\]
Moreover, suppose that 
\Be
\|\bx-\G_{\psi(N)}\bx\|\leq \,C\,\sigma_N(\bx),\quad\forall\;N\geq1,\;\;\bx\in\SX.
\label{Gphia}
\Ee
Then $\psi(N)\geq\,c'\, N^{\beta}$, for some $c'>0$.
\end{theorem}

Concerning this example, we shall also compare the WCGA with the usual TGA, showing that the former performs better when $p>q'$;
see Figure \ref{fig1} below.

Our second application of Theorem \ref{th_Ds} regards the $d$-variate Haar system in $\SX=L^p([0,1]^d)$.
Here, however, we must renorm the space to obtain new results. We consider the Littlewood-Paley renorming, defined as follows.
Let $\cH^d_p=\cH_p\times\ldots\times\cH_p=\{H_n\}_{n\geq0}$ denote the $L^p$-normalized $d$-variate Haar system (ie, the tensor product of the 1-dimensional Haar basis $\cH_p$ in $[0,1]$). Given  
$f=\sum_{n\geq0} c_{n}(f)H_{n}\in L^p$,  we let
\Be
\tri{f}:=\big\|S(f)\big\|_{L^p}, \quad \mbox{where}\quad S(f)=\Big(\sum_{n\geq0}|c_{n}(f)H_{n}|^2\Big)^\frac12.
\label{tri1}
\Ee
This defines an equivalent norm in $L^p([0,1]^d)$, provided $1<p<\infty$.
Our next result gives the optimal growth for the Lebesgue type functions when $1<p\leq 2$.

\begin{theorem}\label{th_haar}
Let $\SX=L^p([0,1]^d)$, $1<p\leq 2$, endowed with the norm $\tri{\cdot}$,
and let $\cD=\cH^d_p$ be the Haar basis.
Let
\[
h(p,d)=\,(d-1)\Big|\frac12-\frac1p\Big|.
\]
Then,  the WCGA satisfies \eqref{oldphi}
with
\[
\oldphi(N)\,=\,\big\lfloor c\,(1+\log N)^{p'h(p,d)}\, N\big\rfloor,
\]
for some $c=c(p,d,\tau)>0$. Moreover, suppose that 
\Be
\tri{f-\G_{\psi(N)}f}\leq \,C\,\sigma_N(f),\quad\forall\;N\geq1,\;\; f\in\SX.
\label{Gpsi_tri}
\Ee
Then $\psi(N)\geq\,c'\,(1+\log N)^{p'h(p,d)}\, N$, for some $c'>0$.
\end{theorem}

\BR
In the case $2<p<\infty$, it is known that an analog of \eqref{Gpsi_tri}, with $\stri{\cdot}$ replaced by the standard $\|\cdot\|_{L^p}$, holds with 
\[
\psi(N)\,\leq\, c\, N^{2/p'};
\]
see \cite[Example 4]{Tem14}. However, we do not know whether this power is optimal.
\ER

The proof of Theorem \ref{th_haar}, presented in \S\ref{S_haar}, is
obtained as the special case $q=2$ of a more general statement for a class of sequence spaces $\fpq$. 
These are defined by norms as in \eqref{tri1}, but with the square function $Sf$ replaced by an $\ell^q$-function
$S_qf$; see  \S\ref{S_fpq} below for precise statements. The spaces $\ell^p(\ell^q)$ and $\fpq$ are discrete analogs of the 
Besov and Triebel-Lizorkin spaces, so our results can be transferred as well to these settings, see Remarks \ref{R_TL} and \ref{R_Besov} below.

Finally, in the last part of the paper we discuss the significance of property $\At$
for suitable classes of \emph{bases} $\cD=\Psi=\{\psi_n\}_{n\geq1}$ in $\SX=L^p$.
As a consequence of our results we shall obtain the following.

\begin{theorem}
\label{th_greedy}
Let $\SX=L^p[0,1]$, $1<p<\infty$, and let $\cD=\Psi$ be a greedy basis with respect to the TGA.
Let 
\Be
\label{alp}
\al(p)=1\quad \mbox{if {\small $1<p\leq 2$}},\mand \al(p)=2/p' \quad \mbox{if {\small $2<p<\infty$}}.
\Ee
Then, there exists a universal constant $C\geq1$ such that 
\[
\big\|f-\G_{cN^{\al(p)}}f\big\|_p\leq \,C\,\sigma_N(f),\quad\forall\;N\geq1,\;\; f\in L^p,
\] 
for some $c=c(p,\tau)>0$. 
\end{theorem}

In particular, when $1<p\leq 2$, greedy bases with respect to the TGA are \emph{almost greedy} with respect to the WCGA 
(in the sense  
of Definition \ref{ID1}). For $p>2$, however, we do not know whether the power $\al(p)=2/p'$ may be improved. 
These resuls are discussed in \S \ref{S_Nik}, together with various additional statements under weaker assumptions in $\Psi$ (such as almost-greedy, quasi-greedy, etc...), where the best exponents $r$ in property $\At$ are found for each class of bases. 
We remark that, in the case of $\SX=L^p$ with $p>2$, we do not know whether, for some basis (or dictionary) $\cD$, 
the Lebesgue inequality for the WCGA may hold with $\oldphi(N)=O(N)$.

\section{Preliminaries}
\setcounter{equation}{0}\setcounter{footnote}{0}
\setcounter{figure}{0}

\subsection{Norming functionals and distance to subspaces}

Let $(\SX,\|\cdot\|)$ be a Banach space over $\SK=\SR$ or $\SC$.
Let $\rho(t)$ be the associated modulus of smoothness, given in \eqref{rho}.
Recall that $\|\cdot\|$ is called \emph{uniformly smooth} if $\rho(t)=o(t)$ as $t\to0$.
Recall also from \eqref{F_f} that $F_f\in \SX^*$ denotes the norming functional of a vector $f\in\SX\setminus\{0\}$.
In all our proofs below, the functional $\rho(t)$ will only appear via the following inequality.
\

\begin{proposition}\label{P_dt}
For all $f,g\in\SX$ with $\|f\|=\|g\|=1$  it holds
\Be
0\leq \|f+tg\|-\|f\|-t\Re[F_f(g)]\leq 2\rho(t),\quad t\in\SR.
\label{dt}
\Ee
In particular, if $\rho(t)=o(t)$, then $\Re[F_f(\cdot)]$ is the Fr\'echet derivative of $\|\cdot\|$ at $f$.
\end{proposition}
\begin{proof}
The left inequality in \eqref{dt} follows easily from
\Be
\|f+tg\|\geq |F_f(f+tg)|\geq \Re[F_f(f+tg)]=\|f\|+t\,\Re[F_f(g)],\quad t\in\SR.
\label{auf0}
\Ee 
Define
\[
\sigma(t,f,g):=\|f+tg\|+\|f-tg\|-2\|f\|\geq0, \quad \forall\;t\in\SR.
\]
Then, using \eqref{auf0} with $t$ replaced by $-t$ we have
\[
\|f+tg\|=2\|f\|-\|f-tg\| + \sigma(t,f,g)\leq 2\|f\|-\|f\|+t\Re[F_f(g)]+ \sigma(t,f,g),
\]
and therefore
\[
\|f+tg\|-\|f\|-t\Re[F_f(g)]\leq \sigma(t,f,g)\leq 2\rho(t), \quad t\in\SR.\]
This establishes the upper bound in \eqref{dt}.
\end{proof}

By homogeneity one deduces
\begin{corollary}\label{C_dt}
For every (non-null) $f,g\in\SX$ with $\|g\|=1$  it holds
\Be
0\leq \|f+tg\|-\|f\|-t\Re[F_f(g)]\leq 2\|f\|\rho(t/\|f\|),\quad t\in\SR.
\label{dtx}
\Ee
\end{corollary}

An interesting application of \eqref{dtx} gives the following result.

\begin{proposition}\label{P_dist}
Suppose that $2\rho(t)\leq \ga \, t^q$, for some $\ga>0$ and $q>1$. Then, there exists $c_1=c_1(\ga,q)>0$ such that
\Be
\dist(f,[\phi])\leq \|f\|\,\big(1-c_1|F_f(\phi)|^{q'}\big),
\label{distFx}
\Ee
for all $f,\phi\in\SX\setminus\{0\}$ with $\|\phi\|=1$.
\end{proposition}
\begin{proof}
Let $\nu={\overline{\sgn}}{F_f(\phi)}$ and $\la>0$. Then Corollary \ref{C_dt} implies that
\Bea
\|f-\la\nu\phi\| & \leq &  \|f\|-\la\Re[F_f(\nu\phi)] +\ga\|f\|(\la/\|f\|)^q\nonumber \\
&  = & \|f\|\Big(1-\tfrac{\la}{\|f\|}|F_f(\phi)| +\ga(\tfrac{\la}{\|f\|})^q\Big).
\label{aux2b}
\Eea
Now, it is easily checked that the 1-dimensional function $h(\la)=\ga\la^q-\la F$ 
reaches a minimum at $\la_1=[F/(\ga q)]^{q'-1}$, and that
\[
\min_{\la>0}h(\la)=-c_1F^{q'},\quad \mbox{with}\quad c_1=\frac1{q'(\ga q)^{q'-1}}.
\] 
So, minimizing over $\la>0$ in \eqref{aux2b} one obtains \eqref{distFx}.
\end{proof}

Thus, recalling Definition \ref{D_Ds}, we recover the following assertion from \S1.

\begin{corollary}
Suppose that $\rho(t)\leq \ga \, t^q$, for some $\ga>0$ and $q>1$.
Then, for every dictionary $\cD$, the pair $(\SX,\cD)$ has the property $\D(s, c_1)$ with $s=q'$ and some $c_1=c_1(\ga,q)>0$.
\end{corollary}

\BR
Observe that $\rho(t)\leq \ga \, t^q$ can only hold for $q\leq 2$, and hence $q'\geq2$. However, property $\D(s,c_1)$ may hold in some cases with $s$ close to 1.
For instance, if $\SX=\ell^p$ with $p>2$ and $\cD$ is the canonical basis, then it is not hard to verify that $(\SX,\cD)$ has property $\D(s,c_1)$ with $s=p'$
and $c_1=1/p$; see Proposition \ref{P_distlpq} below for a more general class of such examples.
\ER

\BR
Observe that the property $\D(s,c_1)$ can also be written as
\[
\Big|F_f(\phi)\Big|\leq c_{1}^{-1}\,\Big(1-\dist(f,[\phi])\Big)^{1/s},
\quad\forall\;\|f\|=1,\;\;\phi\in\cD.
\]
One may ask whether this property could hold 
for some non-uniformly smooth Banach space $(\SX,\|\cdot\|)$ and some dictionary $\cD$.
\ER

\

Below we shall also use the following known lemma.

\begin{lemma}\label{L_Fz}
Assume that the norm $\|\cdot\|$  is G\^ateaux differentiable in $\SX$. Consider a vector $f\in\SX$, a finite dimensional subspace $\cY\subset\SX$,
and an element $g\in\cY$ such that $\|f-g\|=\dist(f,\cY)>0$. Then
\Be
F_{f-g}(h)=0, \quad \forall\;h\in \cY. 
\label{Fz}
\Ee
\end{lemma}
\begin{proof}
Let $h\in\cY$ and let $\nu={\overline{\sgn}}{F_{f-g}(h)}$. Then, for all $\la>0$ we have
\Beas
\dist(f,\cY) & \leq &  \|f-g-\la \nu h\|\leq \|f-g\|-\la \Re[F_{f-g}(\nu h)] + o(\la)\\
 & = & \dist(f,\cY) -\la |F_{f-g}(h)| + o(\la).
\Eeas
Thus, $|F_{f-g}(h)|\leq o(\la)/\la$, which letting $\la\searrow0$ implies that $F_{f-g}(h)=0$.
\end{proof}

\BR
As a consequence of Lemma \ref{L_Fz}, the vectors 
$\phi_{i_1},\phi_{i_2}, ...$ chosen by the WCGA are always linearly independent.
\ER

\subsection{Properties $\Au$ and $\At$ for bases}
\label{S_A2A3}

Suppose that $\cD=\Psi=\{\psi_n\}_{n=1}^\infty$ is a (normalized) Schauder basis in $\SX$. 
We denote its dual system by $\Psi^*=\{\psi^*_n\}_{n\geq1}$.
We briefly discuss the meaning of properties $\Au$ and $\At$ in this case, and relate 
them with more familiar concepts from the theory of thresholding greedy algorithms.

In the first result we use the standard notation 
\[
S_A(f):=\sum_{n\in A}\psi^*_n(f)\psi_n, \quad f\in \SX,
\]
for each finite set $A\subset\SN$. The next lemma is immediate from the definitions.

\begin{lemma}
Let $\Psi$ be a (normalized) Schauder basis in $\SX$. Then, for each $N\geq1$,  $\Sigma_N$ satisfies the property $\Au$ with parameter
\Be
U=k_N:=\sup_{|A|\leq N}\|S_A\|.
\label{kN}
\Ee
In particular, $(\SX,\Psi)$ has the property $\Au(U)$ if and only if $\Psi$ is an unconditional basis, with suppression unconditionality constant $U$.
\end{lemma}


The second result relates property $\At$ with the right-democracy function of the dual basis $\Psi^*$.
We use the convenient notation
\[
\bbone^*_{\bfe A}:=\sum_{n\in A} \e_n\psi^*_n,
\]
when $A\subset\SN$ is a finite set, and $\bfe=(\e_n)_{n=1}^\infty\subset\SK$ is such that $|\e_n|=1$ for all $n$.

\begin{lemma} \label{L_A3}
Let $\Psi$ be a (normalized) Schauder basis in $\SX$. Then $(\SX,\Psi)$ satisfies property {\rm $\At(r, V)$}
if and only if
\Be
\big\|\bbone^*_{\bfe A}\big\|_{\SX^*}\leq V\,|A|^r 
\label{1A3}
\Ee
for all finite $A\subset\SN$, and all $\bfe=(\e_n)_{n=1}^\infty\subset\SK$  with $|\e_n|=1$. 
\end{lemma}
\begin{proof}
Assume first \eqref{1A3}, and take sets $A\subset B$ and scalars $a_n\in\SK$. If $\e_n=\lsgn a_n$ then
\[
\sum_{n\in A}|a_n|=\bbone^*_{\bfe A}\big(\Ts\sum_{n\in B} a_n\psi_n\big)\leq \|\bbone^*_{\bfe A}\|_{\SX^*}\,\big\|\sum_{n\in B} a_n\psi_n\big\|\leq V|A|^r\big\|\sum_{n\in B} a_n\psi_n\big\|.
\]
Conversely, assume property $\At(r,V)$, and pick $A$ and $\bfe$. If $f\in\SX$ has a finite expansion with respect to $\Psi$ we have
\[
|\bbone_{\bfe A}^*(f)|\leq \sum_{n\in A}|\psi^*_n(f)|\leq V\,|A|^r\,\big\|\sum_{n=1}^\infty\psi^*_n(f)\psi_n\big\|=V\,|A|^r\,\|f\|.
\]
Therefore, $\|\bbone^*_{\bfe A}\|_{\SX^*}\leq V|A|^r$, and \eqref{1A3} holds.
\end{proof}

\section{Proof of Theorem \ref{th_Ds}}
\setcounter{equation}{0}\setcounter{footnote}{0}
\setcounter{figure}{0}

We use the definition of WCGA given in \S1. 
We shall follow closely the arguments given in \cite[\S3]{Tem14}; see also  \cite[\S 8.7]{Tem18}.
As in that paper, the proof is split into two main steps. 

\subsection{The iteration inequality}\label{S_step1}

In the first step we prove the following variant of 
\cite[Theorem 2.3]{Tem14}; see also \cite[Theorem 8.7.12]{Tem18}.
The proof makes use of the property $\D(s,c_1)$, and has been rewritten to simplify some steps from \cite{Tem14}.

\begin{theorem}
Let $K\geq 1$ and assume that 
\Benu
\item[(i)] $(\SX,\cD)$ has property $\D(s,c_1)$, for some $s>1$
\item[(ii)] $\Sigma_K$ has property {\rm $\At(r,V)$}, for some $r\leq 1$. 
\Eenu
Then, for all $f\in\SX$, $\Phi\in\Sigma_K$ and all $m, M\geq0$ 
it holds
\Be
\|f_{m+M}\|\leq e^{ -\frac{c_2\, M}{K^{rs}}}\,\|f_m\|\,+\,2\,\|\Phi-f\|,
\label{fmM}
\Ee
where $c_2=\, c(s)\,c_1\,\tau^{s}\,V^{-s}$, for some $c(s)>0$.
\end{theorem}

\begin{proof} 
The case $f_{m+M}=0$ is trivial, so we assume $\|f_{m+M}\|>0$.
If $n\in[m, m+M)$ then also $\|f_n\|>0$.  From the definition of the algorithm and property $\D(s,c_1)$, 
\Be
\|f_{n+1}\|\leq \dist(f_n, [\phin])\,\leq \,\|f_n\|\,\Big(1-c_1|F_{f_n}(\phin)|^{s}\Big).
\label{fnF}
\Ee
When $\Phi=\sum_T a_i\phi_i\in \Sigma_K$, we use the notation  
\[
\Phi_A:=\sum_{A} a_i\phi_i, \quad\mbox{for each}\quad  A\subset T.
\]
We also denote
\Be
\Ga_n:=\supp\G_n(f)=\{i_1,\ldots, i_n\},\mand T_n=T\setminus \Ga_n.
\label{GanTn}
\Ee
By Lemma \ref{L_Fz} and the definition of the algorithm we have
\Be
\tau|F_{f_n}(\Phi_A)|=\tau|F_{f_n}(\Phi_{A\cap T_n})|\leq \sum_{A\cap T_n}|a_i|\,|F_{f_n}(\phin)|.
\label{auf1}
\Ee
By property $\At$,
\[
\sum_{A\cap T_n}|a_i|\leq V\,|A|^r\,\|\Phi-\G_n f\| \leq V\,|A|^r\,\big(\|\Phi-f\|+\|f_n\|\big).
\]
Therefore, inserting these estimates into \eqref{fnF} we obtain
\Be
\|f_{n+1}\|\,\leq \,\|f_n\|\,\Big[1-c_1\Big(\frac{\tau\, |F_{f_n}(\Phi_A)|}{V\,|A|^r\,(\|\Phi-f\|+\|f_n\|)
}\Big)^{s}\Big],
\label{auf2}
\Ee
for all (non-empty) sets $A\subset T$.
One wants to pick the best possible set $A$ in \eqref{auf2}.
For later estimates we shall allow sets $A\subset T_{k}$, for a fixed $k\in [0,m]$, and define
\[
B=T_{k}\setminus A.
\] 
In this theorem it will suffice to consider $k=0$ and $A=T$ (so $B=\emptyset$), but we keep on with the general case.
Since $\Ga_{k}\subset\Ga_n$,  using Lemma \ref{L_Fz}
we have
\Beas
|F_{f_n}(\Phi_A)| & = & |F_{f_n}(\Phi_A+\Phi_{\Ga_{k}\cap A}-\G_nf)|\,=\, |F_{f_n}(\Phi-\Phi_{B}-f+f_n)|\\
& \geq &  \|f_n\|-\|\Phi-f\|-\|\Phi_B\|.
\Eeas

Inserting the above inequalities into \eqref{auf2}, and denoting $c_3=c_1(\tau/(2V))^{s}$, we obtain
\Beas
\|f_{n+1}\| & \leq & \|f_n\|\,\Big(1-c_3\,\frac{[\|f_n\|-\|\Phi-f\|-\|\Phi_B\|]^{s}_+}{[2^{-1}|A|^r\,(\|\Phi-f\|+\|f_n\|)]^{s}}\Big)\\
& \leq & \|f_n\|\,\Big(1-c_3\frac{[\|f_n\|-\|\Phi-f\|-\|\Phi_B\|]^{s}_+}{[|A|^r\,\|f_n\|]^{s}}\Big),
\Eeas
where the last inequality is trivial when $\|f_n\|\leq \|\Phi-f\|+\|\Phi_B\|$, and in the complement case follows from $\|f_n\|\geq \|\Phi-f\|$. We now use the elementary inequality
\[
(1-u)_+^{s}\geq \dt_s(1-2u),
\quad u>0,
\]
which holds for some $\dt_s>0$. Applying this with $u=(\|\Phi-f\|+\|\Phi_B\|)/\|f_n\|$, and letting $c_2=c_3\dt_s$, we obtain
\Be
\|f_{n+1}\|\,\leq\, \|f_n\|\,\Big(1-\frac{c_2}{|A|^{rs}}\Big) \,+ \,\frac{2\,c_2}{|A|^{rs}}\,(\|\Phi-f\|+\|\Phi_B\|).
\label{fn1}
\Ee
This expression can be rewritten as
\[
\|f_{n+1}\|-2(\|\Phi-f\|+\|\Phi_B\|)\,\leq\, \Big(\|f_n\|-2\,(\|\Phi-f\|+\|\Phi_B\|)\Big)\,\Big(1-\frac{c_2}{|A|^{rs}}\Big).
\]
Denoting $\beta=c_2/|A|^{rs}$,  one also has 
\[
\Big(\|f_{n+1}\|-2(\|\Phi-f\|+\|\Phi_B\|)\Big)_+\,\leq\, (1-\beta)\,\Big(\|f_n\|-2\,(\|\Phi-f\|+\|\Phi_B\|)\Big)_+\,,
\]
after checking the trivial negative cases. One can now iterate this inequality for $m\leq n<m+M$ to obtain
\Beas
\|f_{m+M}\|-2(\|\Phi-f\|+\|\Phi_B\|)& \leq &  (1-\beta)^M\,\Big(\|f_m\|-2\,(\|\Phi-f\|+\|\Phi_B\|)\Big)_+
\\
& \leq & (1-\beta)^M\,\|f_m\|,
\Eeas
and therefore
\Be
\|f_{m+M}\|\leq (1-\beta)^M\,\|f_m\|\,+\,2(\|\Phi-f\|+\|\Phi_B\|).
\label{AB}
\Ee
Since $1-\beta\leq e^{-\beta}$ and $\beta\geq c_2/K^{rs}$, letting $B=\emptyset$ in \eqref{AB} gives \eqref{fmM}. 
\end{proof}

\BR
The above proof actually shows the validity of the following more general inequality 
\Be
\|f_{m+M}\|\leq e^{-{c_2M}/{|A|^{rs}}}\,\|f_m\|\,+\,2(\|\Phi-f\|+\|\Phi_B\|),
\label{AB1}
\Ee
for any sets $A\subset T_{k}$ and $B=T_{k}\setminus A$, and any $0\leq k\leq m$. 
This is the analog of the corresponding expression in \cite[(3.7)]{Tem14}; see also \cite[(8.7.22)]{Tem18}.
Observe that \eqref{AB1} also  makes sense for $A=\emptyset$, in which case the inequality 
follows trivially from
\[
\|f_{k}\|\leq \|f-\Phi_{\Ga_{k}}\|=\|f-\Phi+\Phi_B\|\leq \|\Phi-f\|+\|\Phi_B\|.
\]
\ER

\subsection{Proof of Theorem \ref{th_Ds}}\label{S_step2}

The second part of the proof of Theorem \ref{th_Ds} is then carried exactly as in the paper \cite{Tem14}; see also \cite[pp. 433-440]{Tem18}.
It is based on appropriate choices of sets $A_j$ and $B_j$ in inequality \eqref{AB1}, at various inductive steps.
We only remark that the quantities $|A_j|^{rq'}$ that appear in \cite{Tem14} can all be replaced by $|A_j|^{rs}$, according to 
inequality \eqref{AB1}, as it is only via this inequality that the uniform smoothness of the norm (or the property $\D(s,c_1)$) is used in this proof. All the arguments in \cite{Tem14} can then be applied verbatim, so we do not write down the details.

\

\ProofEnd

\section{The WCGA in the spaces $\ell^p(\ell^q)$ }\label{S_lpq}
\setcounter{equation}{0}\setcounter{footnote}{0}
\setcounter{figure}{0}

In this section we illustrate the performance of the WCGA in the space $\SX=\ell^p(\ell^q)$, when $1<p,q<\infty$.
This is the set of all sequences $\bx=(x_{j,k})_{j,k=1}^\infty\subset \SK$ such that
\Be
\|\bx\|:=\Big[\sum_{j=1}^\infty\Big(\sum_{k=1}^\infty|x_{j,k}|^q\Big)^{\frac pq}\Big]^\frac1p<\infty.
\label{lpq}
\Ee
Throughout this section we let $\cD=\{\be_{j,k}\}_{j,k\geq1}$, the canonical basis in the space $\SX$.

\subsection{Norming functionals in $\ell^p(\ell^q)$}
It will be convenient to use the following notation
\[
\Dt_j(\bx):=\Big(\sum_{k=1}^\infty|x_{j,k}|^q\Big)^{\frac 1q},
\]
so that $\|\bx\|=(\sum_{j\geq 1}\Dt_j(\bx)^p)^{1/p}$.

\begin{lemma}
\label{L_Flpq}
When $\bx\in\ell^p(\ell^q)\setminus\{0\}$, the norming functional $F_\bx$ is given by
\Be
F_\bx(\by)=\frac1{\|\bx\|^{p-1}}\,\sum_{j=1}^\infty \Dt_j(\bx)^{p-q}\sum_{k=1}^\infty|x_{j,k}|^{q-2}\bar{x}_{j,k}y_{j,k}.
\label{Flpq}
\Ee
\end{lemma}
\begin{proof}
It is straightforward to check that $F_\bx(\bx)=\|\bx\|$. So we only need to verify that $|F_\bx(\by)|\leq \|\by\|$. 
This follows from a double use of H\"older's inequality
\Beas 
|F_\bx(\by)| & \leq & \|\bx\|^{-(p-1)}\sum_{j}\Dt_j(\bx)^{p-q}\,\Big(\sum_k|x_{j,k}|^{(q-1)q'}\Big)^{1/{q'}}\,\Dt_j(\by)\\
& = & \|\bx\|^{-(p-1)}\sum_{j}\Dt_j(\bx)^{p-1}\,\Dt_j(\by)\\
& \leq & \|\bx\|^{-(p-1)}\Big(\sum_{j}\Dt_j(\bx)^{(p-1)p'}\Big)^{1/p'}\,\|\by\|=\|\by\|.
\Eeas
\end{proof}

We shall apply the functionals $F_\bx$ to the elements $\be_{j,k}$ of the canonical basis.
\begin{corollary}
\label{c_Flpq}
If $\bx\in\ell^p(\ell^q)\setminus\{0\}$ and $j,k\in\SN$, then
\Be
\big|F_\bx(\be_{j,k})\big|\,=\,\frac{\Dt_j(\bx)^{p-q}\,|x_{j,k}|^{q-1}}{\|\bx\|^{p-1}}\,.
\label{clpq}
\Ee
\end{corollary}

As special cases notice that
\Bi
\item if $\supp\bx\subset\{j_0\}\times \SN$, then $|F_\bx(\be_{j,k})|=\,\dt_{j,j_0}\,|x_{j_0,k}|^{q-1}/\|\bx\|^{q-1}$.
\item if $|\supp\bx\cap(\{j\}\times \SN)|\leq 1$, for all $j$, then $|F_\bx(\be_{j,k})|=\,|x_{j,k}|^{p-1}/\|\bx\|^{p-1}$.
\Ei
In particular, for vectors $\bx$ in these two cases, the WCGA (with $\tau=1$) coincides with the usual TGA; that is, after $N$ steps of the algorithm it picks the 
$N$ largest coefficients, and the remainder $\bx_N$ is the projection of $\bx$ on the subspace spanned by the 
remaining basis vectors.

\subsection{Performance of WCGA in $\ell^p(\ell^q)$}
 
Here we compute the relevant parameters in (i)-(iii) of Theorems \ref{th_tem14} and \ref{th_Ds}.
Since the canonical basis is unconditional we see that property $\Au$ holds with $U=1$.
Property $\At$ is given by the next lemma, which follows immediately from the inclusion $\SX\hookrightarrow \ell^{\max\{p, q\}}$.

\begin{lemma}\label{Lpq2}
For every finite set $A\subset\SN\times\SN$ and every $\bx\in\SX$ we have
\[
\sum_{(j,k)\in A}|x_{j,k}|\leq |A|^r\,\|\bx\|,\quad \mbox{with $r=\max\{\frac1{p'},\frac1{q'}\}$.}
\]
 In particular, property $\At$ holds with parameters $r$ and $V=1$.
\end{lemma}

The power of uniform smoothness for the standard norm \eqref{lpq} in  $\ell^p(\ell^q)$
can be estimated as follows; see \cite[Proposition 17]{figiel76} or \cite[Theorem 3.5]{kaz13}.

\begin{lemma}\label{Lpq1}
If $1<p,q<\infty$, then $\rho(t)\leq \ga\,t^\sigma$ with $\sigma=\min\{p,q,2\}$.
\end{lemma}

If at this point one applies Theorem \ref{th_tem14}, a simple computation easily leads to
the following.

\begin{corollary}\label{c_beta}
Let $\beta=\max\{\frac{p'}{q'}, \frac{q'}{p'}\}$. Assume also that 
\Be
\min\{p,q\}\leq 2.
\label{pq2}
\Ee
 Then there are constants $C_1$ and $C_2$ such that
\Be
\|\bx-\G_{C_1\,N^\beta}(\bx)\|\leq \,C_2\,\sigma_N(\bx),\quad \forall\, N\geq1,\;\;\bx\in\SX.
\label{Nbeta}
\Ee 
\end{corollary}

\BR \label{R_lpq}
If one attempts to use Theorem \ref{th_tem14} in the case $\min\{p,q\}>2$, then $\sigma =2$ in Lemma \ref{Lpq1},
and one only obtains $\|\bx-\G_{C_1\,N^{\beta_1}}(\bx)\|\leq\, C_2\,\sigma_N(\bx)$ with
\[
\beta_1=\max\{\tfrac2{p'},\tfrac2{q'}\}>\max\{\tfrac{p'}{q'}, \tfrac{q'}{p'}\}.
\] 
As we shall see below, this result can be improved using instead Theorem \ref{th_Ds}.
\ER

\subsection{An improvement of the previous bound}

We begin by establishing the validity of property $\D(s, c_1)$.

\begin{proposition}\label{P_distlpq}
Let $1<p,q<\infty$, and $s=\max\{p',q'\}$. Then there exists $c_1=c_1(p,q)>0$ such that, 
\Be
\dist(\bx,[\be_{j,k}])\leq \|\bx\|\,\big(1-c_1|F_\bx(\be_{j,k})|^{s}\big),
\label{distlpq}
\Ee
for all $(j,k)\in\SN\times\SN$ and all $\bx\in \SX=\ell^p(\ell^q)\setminus\{0\}$.
In particular, the canonical basis in $\ell^p(\ell^q)$ has the property  $\D(s, c_1)$ with $s=\max\{p',q'\}$.
\end{proposition}
\begin{proof}
We may assume that $\|\bx\|=1$ and that $j=k=1$. Using the notation $\Dt_j=\Dt_j(x)=(\sum_k|x_{j,k}|^q)^{1/q}$, we can write
\Be
\dist(\bx,[\be_{1,1}])^p = \sum_{j=2}^\infty \Dt_j^p+ \Big(\Dt_1^q-|x_{1,1}|^q\Big)^{\frac pq} = 1-\Dt_1^p\,+\,\Dt_1^p\Big(1-|\tfrac{x_{1,1}}{\Dt_1}|^q\Big)^{\frac pq}.
\label{dist_aux1}
\Ee
We shall often make use of the following elementary inequality.
\begin{lemma}\label{1ua}
Let $\al>0$. Then
\[
\,\big(1-u\big)^\al
\,\leq\, 1-\min\{1,\al\} u, \quad u\in[0,1].
\]
\end{lemma}

%

In the last expression of  \eqref{dist_aux1} we apply Lemma \ref{1ua}, with $u=|x_{1,1}/\Dt_1|^q$ and $\al=p/q$, so if $c=\min\{1,p/q\}$ we obtain
\Be
\dist(\bx,[\be_{1,1}])^p \leq  1-\,c\,\Dt_1^p\,|\tfrac{x_{1,1}}{\Dt_1}|^q
\,=\,1-\,c\,|F_\bx(\be_{1,1})|\,|x_{1,1}|,
\label{dist_lpq_aux}
\Ee
the last equality due to \eqref{clpq}.
Letting $s=\max\{p',q'\}$, note that $\min\{p-1,q-1\}= s'-1$.
So, using that both $\Dt_1$ and $|x_{1,1}|/\Dt_1$ are $\leq 1$, it follows that
\[
|F_\bx(\be_{1,1})|^{s-1}=\Dt_1^{\frac{p-1}{s'-1}}\,\Big|\frac{x_{1,1}}{\Dt_1}\Big|^{\frac{q-1}{s'-1}}
\,\leq \,|x_{1,1}|.
\]
Therefore, we have shown
\[
\dist(\bx,[\be_{1,1}])^p \leq 1-c|F_\bx(\be_{1,1})|^{s}.
\]
Finally, using again Lemma \ref{1ua} we obtain \eqref{distlpq} with $c_1=c/p=\min\{\frac1p,\frac1q\}$.
\end{proof}

Thus, Theorem \ref{th_Ds} will produce the following improvement over Corollary \ref{c_beta}.

\begin{corollary}\label{c_beta2}
Let $1<p,q<\infty$, and let $\beta=\max\{\frac{p'}{q'}, \frac{q'}{p'}\}$.
 Then there are constants $C_1$ and $C_2$ such that
\Be
\|\bx-\G_{C_1N^\beta}(\bx)\|\leq C_2\sigma_N(\bx),\quad \forall\, N\geq1,\;\;\bx\in\SX.
\label{Nbeta2}
\Ee 
\end{corollary}

This proves the first assertion in Theorem \ref{th_lpq}.

\subsection{Optimality of the bound in \eqref{Nbeta2}}
We now prove the last assertion in Theorem \ref{th_lpq}.
The lower bounds will be obtained by testing with suitable examples.
For simplicity we assume $\tau=1$ and $C=1$.

\begin{proposition}\label{P_psi}
Let $1<p,q<\infty$. Suppose that 
\Be
\|\bx-\G_{\psi(N)}\bx\|\leq \sigma_N(\bx),\quad\forall\;N\geq1,\;\;\bx\in\SX.
\label{Gphi}
\Ee
Then, there exists $c_{p,q}>0$ such that
\Be
\psi(N)\geq\,c_{p,q}\, N^{\beta},\quad\mbox{with}\quad \beta=\max\{\tfrac{p'}{q'}, \tfrac{q'}{p'}\}.
\label{philow}
\Ee 
\end{proposition}
\begin{proof}
Pick $\bx=\bbone_A+n^\al\bbone_B$, where
\Bi
\item $|A|=m$ and $|A\cap(\{j\}\times\SN)|=1$, for $1\leq j\leq m$, and 0 otherwise.
\item $|B|=n$ and $B\subset\{m+1\}\times\SN$.
\item $\al=p'(\frac1{q'}-\frac1{p'})$.
\Ei
Observe that with this choice we have
\Be
|F_\bx(\be_{j,k})|=\|\bx\|^{-(p-1)}, \quad (j,k)\in A\cup B.
\label{F1}
\Ee
Indeed, if $(j,k)\in A$ then $x_{j,k}=\Dt_j(\bx)=1$ and \eqref{F1} follows from \eqref{clpq}. If $(j,k)\in B$ then $j=m+1$ and
\[
\Dt_{m+1}(\bx)^{p-q}|x_{m+1,k}|^{q-1}= n^{\al(p-q)}\,n^{\frac{p-q}q}\,n^{\al(q-1)}=1.
\]
In the case $p'\geq q'$, for each $1\leq \ell\leq m$ we shall pick $\G_\ell\bx=\bbone_{A_\ell}$, where $A_{\ell-1}\subset A_{\ell}\subset A$ with $|A_\ell|=\ell$. Observe that
this is possible because the equality in \eqref{F1} continues to hold when $\bx$ is replaced by each remainder $\bx_\ell=\bx-\G_\ell\bx$. 
So if we let $m=\psi(n)$, then \eqref{Gphi} implies
\[
n^\al\|\bbone_B\|=\|\bx-\G_m\bx\|\leq \sigma_n(\bx)\leq \|\bx-n^\al\bbone_B\|=\|\bbone_A\|=m^{1/p}.
\] 
Since $\|\bbone_B\|=n^{1/q}$, this implies 
\Be
\psi(n)=m\geq n^{p(\al+\frac1q)}=n^{\frac{p'}{q'}},
\label{phipq}
\Ee
where in the last equality we use the expression of $\al$.

In the case $q'\geq p'$ we use the same example, except that we take $|B|=2n$.
For each $1\leq \ell\leq n$, this time we pick $\G_\ell\bx=n^\al\bbone_{B_\ell}$ where $B_{\ell-1}\subset B_{\ell}\subset B$ with $|B_\ell|=\ell$. 
To justify this choice at each step, observe that, if $(j,k)\in B\setminus B_{\ell}$ then,
\Beas
\|\bx_{\ell}\|^{p-1}|F_{\bx_\ell}(\be_{j,k})| & = & \Dt_{m+1}(\bx_\ell)^{p-q}\,n^{\al(q-1)}= n^{\al(p-1)}\,(2n-\ell)^{\frac{p-q}q}\\
 & = & \Big(\frac{2n-\ell}n\Big)^{\frac pq-1}\geq1,
\Eeas
while for $(j,k)\in A$ we have $\|\bx_{\ell}\|^{p-1}|F_{\bx_\ell}(\be_{j,k})|=1$. So, if we let $n=\psi(m)$, from \eqref{Gphi} we have
\[
m^{1/p}=\|\bbone_A\|\leq\|\bx-\G_n\bx\|\leq \sigma_m(\bx)\leq \|\bx-\bbone_A\|=n^\al\|\bbone_B\|=2^{\frac1q}n^{\al+\frac1q},
\]
which implies 
\Be
\psi(m)=n \geq \,c_{p,q}\,m^{\frac1{p(\al+1/q)}}\,=\,c_{p,q}\,m^{\frac{q'}{p'}}.
\label{phiqp}
\Ee
Combining \eqref{phipq} and \eqref{phiqp} one obtains \eqref{philow}.
\end{proof}

\Proofof{Theorem \ref{th_lpq}}
Combine Corollary \ref{c_beta2} and Proposition \ref{P_psi}.
\ProofEnd

\subsection{WCGA vs TGA in $\ell^p(\ell^q)$}

Let $G_n(\bx)$ denote the usual thresholding greedy algorithm (TGA) applied to $\bx$, with regard to the canonical basis in $\SX=\ell^p(\ell^q)$.
The following result is a consequence of \cite[Theorem 5]{KT04}. 

\begin{proposition}\label{p_b}
Let $1<p,q<\infty$. Then there are constants $C_1, C_2$ such that
\Be
\|\bx-G_{C_1N^{b}}(\bx)\|\leq C_2\,\sigma_N(\bx),\quad \forall\; N\geq1,\;\;\bx\in\SX,
\label{Nb}
\Ee
with $b=\max\{\frac pq,\frac qp\}$.
\end{proposition}
\BR
The exponent $b=\max\{\frac pq,\frac qp\}$ in \eqref{Nb} is best possible, in view of \cite[Theorem 3.2]{KT04}.
\ER

So one may pose the question of when WCGA performs better than TGA, in the sense that 
the power $\beta$ in Corollary \ref{c_beta2} is smaller than the power $b$ in \eqref{Nb}.
Assuming $p> q$, this is equivalent to
\[
\beta=\frac{q'}{p'}\leq \frac pq=b \iff qq'\leq pp'\iff \quad q\geq p' .
\]
If $p< q$, then $\beta\leq b$ iff $p\geq q'$. Thus, overall the WCGA performs better than the TGA, in the sense that $\beta\leq b$, if and only if $q\geq p'$ (or $q=p$).

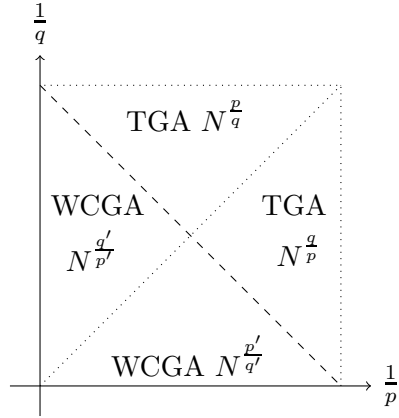
\begin{figure}[ht]
 \centering
{\begin{tikzpicture}[scale=4]
\draw[->] (-0.1,0.0) -- (1.1,0.0) node[right] {$\frac{1}{p}$};
\draw[->] (0.0,-0.1) -- (0.0,1.1) node[above] {$\frac1q$};

\draw[dotted] (1.0,0.0) -- (1.0,1.0)--(0,1);
\draw[dotted] (0,0.0) -- (1.0,1.0);
\draw[dashed] (0.0,1.0) -- (1.0,0);

\draw  (0.25,0.9) node [right] {{\small TGA $N^{\frac pq}$}};
\draw  (0,0.6) node [right] {{\small WCGA}};
\draw  (0.05,0.45) node [right] {{\small $N^{\frac{q'}{p'}}$}};

\draw  (0.2,0.1) node [right] {{\small WCGA $N^{\frac {p'}{q'}}$}};

\draw  (0.7,0.6) node [right] {{\small TGA}};
\draw  (0.75,0.45) node [right] {{\small $N^{\frac{q}{p}}$}};

\end{tikzpicture}
}
\caption{Number of iterations (modulo constants)  of TGA or WCGA to reach $\sigma_N(\bx)$ in $\ell^p(\ell^q)$. }\label{fig1}
\end{figure}

\section{The WCGA in the spaces $\fpq$ }\label{S_fpq}
\setcounter{equation}{0}\setcounter{footnote}{0}
\setcounter{figure}{0}

In this section we consider a class of sequence spaces $\fpq$ related with the 
family of Triebel-Lizorkin spaces.
In the next section we shall specialize to the case $q=2$, and deduce the results 
for the Haar system in $(L^p,\stri{\cdot})$ asserted in Theorem \ref{th_haar}.

Throughout this section we make use of the following notation. 
We fix $d\geq 1$ and let $\sRd$ 
be the set of all dyadic \emph{rectangles} $I\subset[0,1]^d$.
That is, $I=I_1\times\ldots\times I_d$, with each $I_i=2^{-j_i}(k_i+[0,1))$, for some $j_i, k_i\in\SN_0$ and $0\leq k_i<2^{j_i}$, when $i=1,\dots, d$. When $d=1$ we will write $\sR_1=\sD$, so that $\sRd=\sD\times\ldots\times\sD$. 

For fixed $d\geq1$ and $1<p,q<\infty$, we consider the space $\SX=\fpq=\fpq(\sRd)$, defined as  the set of all sequences $\bx=(x_I)_{I\in\sRd}$ such that 
\Be
\|\bx\|=\|\bx\|_{\fpq}:=\Big\|\big(\sum_{I\in\sRd}\big|x_I\bp\big|^q\big)^{1/q}\Big\|_{L^p([0,1]^d)},
\label{norm_fpq}
\Ee
where $\bp=|I|^{-1/p}\bbone_I$ is the $L^p$-normalized characteristic function of $I$.
We shall often use the notation
\[
(S_q\bx)(u):= \big(\sum_{I\in\sRd}\big|x_I\bp(u)\big|^q\big)^{1/q}, \quad u\in[0,1]^d.
\]
Throughout this section, $\cD=\{\be_I\}_{I\in\sRd}$ will be the canonical basis in $\fpq$.

\subsection{Norming funcionals in $\fpq$}
\label{Ssec2}

\begin{lemma}
\label{L_Ffpq}
When $\bx\in\fpq\setminus\{0\}$, the norming functional $F_\bx$ is given by
\Be \label{2-1a}
F_\bx(\by) = \frac1{\|\bx\|^{p-1}}\int [(S_q\bx)(u)]^{p-q} \Big(\sum_{I\in\sRd} |x_{I}|^{q-2} \,{\overline{x_I}}\,y_I \, |\bp(u)|^q\Big) du\,.
\Ee
\end{lemma}

\begin{proof}
It is easy to see that $F_\bx(\bx)=\|\bx\|$, so we only need to verify that $|F_\bx(\by)|\leq \|\by\|$. 
This follows easily from a double use of H\"{o}lder's inequality 
	\begin{eqnarray*}
		\|\bx\|^{p-1}|F_\bx(\by)| &\leq&  
		\int  (S_q\bx)^{p-q} \Big[\sum_{I} |x_I|^{(q-1)q'}\bp^{q}\Big]^\frac1{q'} \,\Big[\sum_I|y_I|^q\bp^q\Big]^\frac1q du\,\\
& = & 	\int  (S_q\bx)^{p-1}\,(S_q\by)\,  du\;\leq\; 	\|\bx\|^{p-1}\,\|\by\|\,.
	\end{eqnarray*}
\end{proof}

If we specialize $F_\bx$ to the elements $\be_I$ of the canonical basis we obtain.
\begin{corollary}
\label{c_Ffpq}
If $\bx\in\fpq\setminus\{0\}$ and $I\in\sRd$, then
\Be
\big|F_\bx(\be_I)\big|\,=\,\frac{|x_I|^{q-1}}{\|\bx\|^{p-1}}\,\frac1{|I|^{q/p}}\int_I|S_q\bx|^{p-q}.
\label{2-2}
\Ee
\end{corollary}

\BR
As a special case, if $\supp\bx$ consists of pairwise disjoint $I$'s, one has
\[
\big|F_\bx(\be_I)\big|\,=\,{|x_I|^{p-1}}/{\|\bx\|^{p-1}}.
\]
So, in this case WCGA and TGA coincide.
\ER

\subsection{Distance to subspaces}
\label{Ssec3}

We study the property $\D(s,c_1)$ in the spaces $\fpq(\sRd)$.
This will make unnecessary to compute
the modulus of smoothness for $\|\cdot\|_{\fpq}$ (which seems not to appear in the literature).

\begin{lemma}
\label{L_Dsfpq}
Let $1<p,q<\infty$, and let $s=s(p,q)=\max\{p', q'\}$. 
Then, for all $\bx\in\fpq\setminus\{0\}$ it holds
\Be \label{2-1}
\dist(\bx,[\be_I])\leq \,\|\bx\|\Big(1-c_{p,q}\big|F_\bx(\be_I)\big|^s\Big),\quad \forall\,I\in\sRd,
\Ee
with $c_{p,q}=\min\{1/p,1/q\}$. In particular, the canonical basis in $\fpq$ has the property  $\D(s, c_1)$ with $s=\max\{p',q'\}$.
\end{lemma}

\begin{proof}
We may assume that $\|\bx\|=1$. Then, 
\[
\dist(\bx,[\be_I])^p = \|\bx-x_I \be_I\|^p 
= \int  (S_q\bx)(u)^p \Big[ 1 - \frac{|x_I\bp(u)|^q}{(S_q\bx)(u)^q}\Big]^{p/q}\,du,
\]
where it is understood that the integral is taken over the set $\{u\mid (S_q\bx)(u)\not=0\}$.
Setting $c:= \min\{ 1, p/q \}$ and applying 
Lemma \ref{1ua} we obtain
\Be
\dist(\bx,[\be_I])^p  \leq    \int  (S_q\bx)^p\Big[ 1 - c\,\Big|\frac{x_I\bp}{S_q\bx}\Big|^q\Big]\,du\,=\,
1\,-\,c\,|x_I|\,\big|F_{\bx}(\be_I)\big|,
\label{dpcases}
\Ee
using in the last step Corollary \ref{c_Ffpq} (and $\|\bx\|=1$). We now separate two cases:

\subsubsection*{a) Case $1< p \leq q$.} Since $|x_I\bp(u)|\leq (S_q\bx)(u)$ and $p-q \leq 0$,
we have
$$
(S_q\bx)(u)^{p-q} \leq |x_I\bp(u)|^{p-q}\,.
$$
Integration over $I$ gives
\[
\int (S_q\bx)^{p-q}\,\bp^q \leq |x_I|^{p-q}\,\int\bp^p=|x_I|^{p-q}.
\]
Then, using \eqref{2-2} we see that
\[
|F_\bx(\be_I)| \,= \, |x_I|^{q-1}\,\int (S_q\bx)^{p-q}\,\bp^q \leq |x_I|^{p-1}.
\]
Thus, we have obtained a lower bound for $|x_I|$, which inserted into \eqref{dpcases} gives
\[
\dist(\bx,[\be_I])^p\,\leq \, 1\,-\, c\,\,\big|F_{\bx}(\be_I)\big|^{1+\frac1{p-1}}\,=\,1\,-\, c\,\,\big|F_{\bx}(\be_I)\big|^{p'}\,.
\]
Finally, since $c=p/q$, a last use of Lemma \ref{1ua} 
 (with $\al= 1/p < 1$) gives
\begin{equation} \label{4-2}
\dist(\bx,[\be_I]) \,\leq \,1\,-\, q^{-1}\,\,\big|F_{\bx}(\be_I)\big|^{p'}\,.
\end{equation}

\subsubsection*{b) Case $q< p <\infty $.} 
 H\"older's inequality with exponents $( p/(p-q),p/q)$ gives
\[
\int (S_q\bx)^{p-q}\bp^q\,\leq\,\Big[\int(S_q\bx)^p\Big]^{\frac {p-q}p}\, \Big[\int\bp^p\Big]^{\frac qp}\,=\,1.
\]
Then, from \eqref{2-2},
$$
|x_I|^{q-1}\,=\,\frac{|F_\bx(\be_I)|}{\int (S_q\bx)^{p-q}\bp^q}\,\geq  \, |F_\bx(\be_I)|
$$
Inserted into \eqref{dpcases} (and using $c=1$) it gives
\[
\dist(\bx,[\be_I])^p\,\leq \, 1\,-\,\big|F_{\bx}(\be_I)\big|^{1+\frac1{q-1}}\,=\,1\,-\,\big|F_{\bx}(\be_I)\big|^{q'}\,.
\]
Hence, a last use of Lemma \ref{1ua} 
 (with $\al= 1/p < 1$) implies
\begin{equation} \label{5-2}
\dist(\bx,[\be_I]) \,\leq \,1\,-\, p^{-1}\,\,\big|F_{\bx}(\be_I)\big|^{q'}\,.
\end{equation}
\end{proof}

\BR
Testing with explicit examples, it is possible to show that the power $s=\max\{p', q'\}$ in \eqref{2-1} cannot
be replaced by any smaller number. 
\ER

\subsection{Property $\At$}
\label{S_Atfpq}

Here, $\Sigma_N$ is the set of $N$-term combinations from the canonical basis of $\SX=\fpq(\sRd)$.
We define the exponent
\[
h=h(p,q;d)=(d-1)\Big(\frac1p-\frac1q\Big)_+.
\] 

\begin{lemma}
\label{L_A3fpq}
Let $1<p,q<\infty$ and $N\geq 1$. Then $\Sigma_N$ satisfies property $\At(r,V)$ with $V=c(\log N)^{h}$ 
and $r=1/p'$, for some constant $c=c(p,q,d)>0$. That is
\Be 
\sum_{I\in A}|x_I|\,\leq\,c\,(1+\log N)^{(d-1)(\frac1p-\frac1q)_+}\,|A|^{1/p'}\,\|\sum_{I\in B}x_I\be_I\|.
\label{A3fpq}
\Ee
for all sets $A\subset B\subset \sRd$ with $|A|\leq N$ and $B$ finite.
\end{lemma}
\begin{proof}
In view of Lemma \ref{L_A3} it suffices to estimate
\[
\sup_{|A|\leq N}\big\|\sum_{n\in A}\be_I\big\|_{(\fpq)^*}\;.
\]
Since $(\fpq)^*=\mathfrak{f}_{p',q'}$, then \eqref{A3fpq} reduces to show that
\Be
\sup_{|A|\leq N}\big\|\sum_{n\in A}\be_I\big\|_{\fpq}\,\leq\, c\, (1+\log N)^{(d-1)(\frac1q-\frac1p)_+}\,N^\frac1p.
\label{hr_fpq}
\Ee
When $q=2$ this was proved in \cite[Proposition 10]{Wo00}, see also \cite[Theorem A]{KPT06}.
A small modification of those proofs gives also the case $q\not=2$.
For completeness, we sketch the arguments in Appendix 1 below.
\end{proof}

\subsection{The WCGA in $\fpq$}

A direct application of Theorem \ref{th_Ds}, with the parameters obtained in Lemmas \ref{L_Dsfpq} and \ref{L_A3fpq}, gives
the following.

\begin{theorem}\label{th_fpq}
Let $\SX=\fpq(\sRd)$, $1<p, q< \infty$, and let $\cD$ be the canonical basis.
Define
\[
h(p,q;d)=\,(d-1)\Big(\frac1p-\frac1q\Big)_+,\mand \al(p,q)=\left\{
\Ba{lll}
1 & {\rm if} &  p\leq q\\
q'/p'& {\rm if} &   q\leq p.
\Ea
\right.
\]
Then,  the WCGA satisfies \eqref{oldphi}
with
\Be
\oldphi(N)\,=\,\big\lfloor c\, (1+\log N)^{p'h(p,q;d)}\, N^{\al(p,q)}\big\rfloor,
\label{phi_fpq}
\Ee
for some $c=c(p,q,d,\tau)>0$.
\end{theorem}
\BR
As in Remark \ref{R_lpq} above, if in the case $\min\{p,q\}>2$ 
we had used Theorem \ref{th_tem14} rather than Theorem \ref{th_Ds},
then one would have obtained a power $2/p'$, which is worse than $\al(p,q)$
(assuming that the estimate for the modulus of smoothness in Lemma \ref{Lpq1} also applies when $\SX=\fpq$).
\ER

\subsection{A lower bound for the WCGA in $\fpq$ when $p\leq q$}

We show in this subsection that, when $p\leq q$, the function $\oldphi(N)=\lfloor c\, (1+\log N)^{p'(d-1)(\frac1p-\frac1q)}\, N\rfloor$ in \eqref{phi_fpq} cannot be replaced by a slower growing one.

\begin{theorem} \label{th_fpq2}
In the conditions of Theorem \ref{th_fpq}, suppose that $p\leq q$ and that
\Be
\big\|\bx-\G_{\psi(N)}\bx\big\|\leq \,C\,\sigma_N(\bx),\quad\forall\;N\geq1,\;\; \bx\in\fpq(\sRd).
\label{Gpsi_fpq}
\Ee
Then $\psi(N)\geq\,c'\,(1+\log N)^{p'(d-1)(\frac1p-\frac1q)}\, N$, for some $c'>0$.
\end{theorem}

\begin{proof}
Fix $n\in\SN$ and $\bm=(m_1,\ldots, m_d)\in\SN^d$, and let $m:=m_1+\ldots+m_d$.
Consider the following sets
\[
A_n=\Big\{I\in\sRd\mid I\subset[0,1/2]^d\mand |I|=2^{-n}\Big\}
\]
and
\[
B_{\bm}=\Big\{I_1\times\ldots\times I_d \in\sRd\mid I_i\subset[1/2,1]\mand |I_i|=2^{-m_i},\;i=1,\ldots,d\Big\}.
\]
Observe that
\Be
|A_n|\,\approx \, n^{d-1}\,2^n
\mand 
\Big(\sum_{I\in A_n}\bp^q(u)\Big)^{1/q}\approx n^{\frac{d-1}q}\,2^{n/p}, \quad u\in (0,1/2)^d,
\label{An1}
\Ee
while
\Be
|B_\bm|\,=\,2^{m-1}
\mand 
\Big(\sum_{I\in B_\bm}\bp^q(u)\Big)^{1/q}\approx \,2^{m/p}, \quad u\in (1/2,1)^d,
\label{Bm1}
\Ee
We build a vector 
\[
\bx=\bx_{A_n}+\bx_{B_\bm}\,:=\,a \sum_{I\in A_n}\be_I\,+\,b \sum_{I\in B_\bm}\be_I,
\] 
for suitable $a,b>0$ to be chosen later. Using the expression for the $\fpq$-norming functionals
in \eqref{2-2}, and the estimates in \eqref{An1}, \eqref{Bm1}, one sees that
\Be
\|\bx\|^{p-1}\,\big|F_\bx(\be_I)\big|=|x_I|^{q-1}\,\int_I (S_q\bx)^{p-q}\,\bp^q\,\approx\,
\left\{\Ba{lll}
a^{p-1}n^{\frac{(d-1)(p-q)}q}, & {\rm if} & I\in A_n\\
b^{p-1}, & {\rm if} & I\in B_\bm.\\
\Ea\right.
\label{FAnBm}
\Ee  
We  now pick $b$ such that
\Be
b^{p-1}\,\approx\,a^{p-1}\,n^{(d-1)(p-q)/q}.
\label{abp}
\Ee
The constants can be adjusted so that the WCGA always chooses
\[
\G_k(\bx)=b\sum_{j=1}^k\be_{I_j},\quad \mbox{when}\quad k\leq |B_\bm|,
\]
for some enumeration $I_1,I_2,\ldots$ of $B_\bm$.
At this point one should notice that \eqref{FAnBm}
continues to hold with $\bx$ replaced by each remainder $\bx-\G_k(\bx)$, as long as 
$k<|B_\bm|$ (so one can still pick elements $I$ from $B_\bm\setminus\{I_1,\ldots, I_k\}$).

Finally, given $N\gg 1$, we select the largest $n$ and the smallest $m$ such that
\Be
|A_n|\leq N\mand |B_\bm|\geq \psi(N).
\label{AnN}
\Ee
Then, $N\approx n^{d-1}2^n$ and $\psi(N)\approx 2^m$. The assumption in \eqref{Gpsi_fpq} gives
\[
\|\bx_{A_n}\|\,\leq\,\|x-\G_{\psi(N)}(\bx)\|\,\leq \,C\,\sigma_N(\bx)\,\leq\,C\,\|\bx_{B_\bm}\|.
\]
By \eqref{An1} and \eqref{Bm1} this implies
\[
 a\,n^{\frac{d-1}q}\,2^{n/p}\approx \|\bx_{A_n}\|\leq\,C\,\|\bx_{B_\bm}\|\,\approx\, b\,2^{m/p}\,\approx\,b\,\psi(N)^{1/p}.
\]
Recalling  \eqref{abp} this amounts to
\[
\psi(N)\,\gtrsim \,(a/b)^p\,n^{\frac{(d-1)p}q}\, 2^n\,
\approx\,n^{\frac{(d-1)p'}{q'}}\, 2^n\,\approx\,[\log N]^{(d-1)(\frac{p'}{q'}-1)}\, N,
\]
using in the last step the relation  $N\approx n^{d-1}2^n$.
This proves the theorem since the exponent in the power of the log can also be written as $p'/q'-1=p'(\frac1p-\frac1q)$.

\end{proof}

\section{The WCGA for the $d$-variate Haar basis in $L^p$}
\label{S_haar}
\setcounter{equation}{0}\setcounter{footnote}{0}
\setcounter{figure}{0}

We recall the definition of the Haar system. Consider the 1-dimensional functions $h_0=\bbone_{[0,1)}$
and $h=\bbone_{[0,1/2)}-\bbone_{[1/2,1)}$. For every dyadic interval $I=2^{-j}(k+[0,1))$ in $[0,1]$, define
\[
h_{I,p}(x)=2^{j/p}h(2^jx-k) ,\mand h_{0,p}(x)= h_0(x).
\]
Then, if $1<p<\infty$, the system $\cH_p=\{h_0, h_{I,p}\}_{I\in \sD}$ is a (normalized) unconditional basis of $L^p[0,1]$.
With a slight abuse of notation we write $\cH_p=\{h_{I,p}\}_{I\in\bD}$ with $\bD=\sD\cup\{0\}$.
The $d$-variate Haar system is then \[
\cH^d_p=\cH_p\times\ldots\times\cH_p=\{H_{I,p}\}_{I\in\bRd},
\]
with $\bRd=\bD\times\ldots\times\bD$, that is,
\[
H_{I,p}(x_1,\ldots,x_d)=\prod_{i=1}^d h_{I_i,p}(x_i), \quad \mbox{if}\quad I=I_1\times\ldots\times I_d\in \bRd.
\]
The system $\cD=\cH^d_p$ is an unconditional basis of $L^p([0,1]^d)$, and moreover, for every function 
$f=\sum_{I} c_{I}(f)H_{I,p}\in L^p$, the expression
\[
\tri{f}:=\big\|S(f)\big\|_{L^p}, \quad \mbox{where}\quad S(f)=\Big(\sum_{I\in\bRd}|c_{I}(f)H_{I,p}|^2\Big)^\frac12,
\]
defines an equivalent norm in $L^p([0,1]^d)$, provided $1<p<\infty$.
Here, the coefficients are given by
\[
c_I(f):=\lan{f}{H_{I,p'}}, \quad I\in\bRd.
\]
Clearly, $(L^p,\stri{\cdot})$ is related with the spaces $\fpt$ from the last section.
More precisely, the coefficient map defines an isometric isomorphism 
\begin{align*}
\big(L^p([0,1]^d),\stri{\cdot &}\big)  \longrightarrow\; \fpt(\bRd)\\
f & \longmapsto \quad \big(c_I(f)\big)_{I\in\bRd}
\end{align*}
with the latter space defined as in \S\ref{S_fpq}, with the only minor change that $\sRd$ is replaced by $\bRd$
(so a few additional terms appear in the norm \eqref{norm_fpq}).
This change does not affect the proofs, and all the results in \S\ref{S_fpq} continue to hold for the spaces $\fpq(\bRd)$.
In particular, Theorem \ref{th_haar} from the introduction becomes a corollary of Theorems \ref{th_fpq} and \ref{th_fpq2},
in the special case $q=2$.

\BR\label{R_TL}
One can use the previous isometric isomorphism to transfer the results for $\SX=\fpq$ in \S\ref{S_fpq}
into results for (univariate) wavelet bases $\Psi=\{\psi_I\}$ in the class of 
Triebel-Lizorkin spaces $F^r_{p,q}$. Indeed, one can define equivalent norms in the latter of the form
\[
\|f\|_{F^r_{p,q}}:=\Big\|\big(\sum_I\, |c_I(f)\bp|^q\big)^\frac1q\Big\|_{L^p},
\]
where $c_I(f)=\lan f{\psi^*_I}=|I|^{-r}\,\lan f{\psi_{I,p'}}$, with $\psi_{I,p'}=\psi_I/\|\psi_I\|_{p'}$. Just remark that in \S\ref{S_fpq} one should modify the index set $\sD$ 
acccording to the underlying space (say, over $\SR$ rather than $[0,1]$). In the $d$-variate case,
the use of tensor wavelet bases, gives rise to an isometry between $\fpq(\sRd)$ and Triebel-Lizorkin classes with dominating mixed smoothness
(sometimes denoted $S^r_{p,q}F$; see e.g. \cite[Th. 1.12]{Tri2019}). 
\ER

\BR\label{R_Besov}
In a similar fashion, one may transfer the results for $\SX=\ell^p(\ell^q)$ in \S\ref{S_lpq} 
into results for wavelet bases in the class of Besov spaces $B^r_{q,p}$
(or $S^r_{q,p}B$ in the $d$-variate case), using the equivalent norms
\[
\|f\|_{B^r_{q,p}}:=\Big[\sum_{j}\Big(\sum_{k}\, |c_{I(j,k)}(f)|^q\Big)^\frac pq\Big]^{1/p},
\]
with $I(j,k)=2^{-j}(k+[0,1))$, $j,k\in\SZ$. 
\ER

\section{Property $\At$ for classes of bases in $L^p$}
\label{S_Nik}
\setcounter{equation}{0}\setcounter{footnote}{0}
\setcounter{figure}{0}

In this section we find optimal exponents for the property $\At$ in certain classes of bases 
of $L^p$, $1<p<\infty$. These classes (greedy, almost greedy, quasi-greedy,...) behave well with respect to the TGA,
and we study the performance with respect to the WCGA. 
Results of this section complement the corresponding results from \cite{Tem14} (see also \cite{Tem18}, Section 8.7.4).

 Below we shall use the following known lemma, which is valid for $L^p$ over an arbitrary ($\sigma$-finite) measure space. We denote by
$L^p_\SK$ the subspace of $\SK$-valued functions in $L^p$.

\begin{lemma}\label{L_rhoC}
Let $1<p<\infty$, and let $\SX=L^p=L^p_\SC$. Then, there exists $\ga=\ga(p)>0$ such that 
\[
\rho_{L^p}(t):=\sup_{\|f\|_p=\|g\|_p=1}\Big(\|f+tg\|_p+\|f-tg\|_p-2\|f\|_p\Big)/2\;\leq\,\ga\, t^q,\quad \forall\, t>0,
\]
with $q=\min\{p,2\}$.
\end{lemma}
\begin{proof}
This result is well-known when $\SX=L^p_\SR$; see e.g. \cite[Vol II, p. 63]{LT}. We remark that in the complex-valued case one has
\[
\rho_{L^p_\SC}(t)\,=\,\rho_{L^p_\SR}(t)\,\leq\,\ga\, t^q,\quad \forall\,t>0\;,
\]
since $L^p_\SC$ can be isometrically embedded into some $L^p_\SR$; see \cite[Example 4.6]{DKSTW}.
\end{proof}

\subsection{Greedy bases}

The univariate Haar system $\cH_p$ is a greedy basis of $L^p[0,1]$, $1<p<\infty$, and it was shown in \cite[Example 4]{Tem14}
that a Lebesgue inequality holds with $\oldphi(N)=O(N)$ if $1<p\leq 2$, and  $\oldphi(N)=O(N^{2/p'})$ if $2\leq p<\infty$. 
Here we generalize this result to any greedy basis $\Psi=\{\psi_n\}_{n\geq1}$ of $L^p$,
as stated in Theorem \ref{th_greedy} from \S1.

%

\

We begin with a computation of the parameters for the $\At$ property.
We assume that $\Psi$ is normalized in $L^p$, and recall that a greedy basis is unconditional and democratic.

\begin{lemma}\label{L_A3Lp}
If $\Psi$ is a greedy basis in $L^p$, then $\At(r,V)$ holds with $r=1/p'$ and some constant $V>0$.  
\end{lemma}
\begin{proof}
Let $\Psi^*=\{\psi^*_n\}_{n=1}^\infty$ be the dual system to $\Psi$ in $L^{p'}$.
In view of Lemma \ref{L_A3} it suffices to show that
\Be
\big\|\sum_{n\in \La}\psi^*_n\big\|_{L^{p'}} \,\leq \, C\, |\La|^{1/p'},
\label{A3Lp}
\Ee
for all finite sets $\La\subset\SN$. We shall use a result from functional analysis \cite[Theorem 4b]{KP} which says that for any
(seminormalized) unconditional basis $\cB=\{b_k\}_{k\geq1}$ of $L^q$, $1<q<\infty$, there is a
subsequence $k_j$,
$j=1,2,\dots$, such that
$$
\big\|\sum_{j=1}^\infty \alpha_{k_j} b_{k_j}\big\|_{L^q}\asymp
\Big(\sum_{j=1}^\infty|\alpha_{k_j}|^q\Big)^\frac1q.
$$
In particular, if $\cB$ is unconditional and democratic, then necessarily
\begin{equation}\label{gb2}
\|\sum_{k\in \La}b_k\|_{q}\asymp |\La|^{1/q},
\end{equation}
with  the constants of equivalency depending at most on $\cB$ and $q$.
Now, if $\Psi$ is a greedy basis in $L^p$, $1<p<\infty$, then it was shown in \cite[Theorem 5.1]{DKKT} that
$\Psi^*$ is also a greedy basis in the dual space $(L^p)^*=L^{p'}$.  Therefore \eqref{A3Lp} is just a consequence of \eqref{gb2} with $q=p'$.
 \end{proof}

\Proofof{Theorem \ref{th_greedy}}
The proof is now a direct application of Theorem \ref{th_tem14}. Indeed, $\Au(U)$ holds because $\Psi$ is unconditional,
while $\At(r=1/p',V)$ was proved in the previous lemma.   By Lemma \ref{L_rhoC} we also have
$\rho(t)\leq \ga t^q$ with $q=\min\{p,2\}$. Therefore, Theorem \ref{th_tem14}  implies
\begin{equation}\label{gb4}
\|f-\G_{C(t,p,\Psi)m^{\al(p)}}(f)\|_p \le C\sigma_m(f,\Psi)_p, 
\end{equation}
 with $\al(p)=rq'$, which agrees with \eqref{alp}.
\ProofEnd

%
%

\subsection{Tensor products of greedy bases}
 We now extend the first part of Theorem \ref{th_haar} from the multivariate Haar system to
an arbitrary tensor product of univariate greedy bases.
  Let $\Psi$ be a normalized basis in $L^p([0,1))$. In the space $L^p([0,1)^d)$ we define
	the system
\[
\Psi^d:=\Psi\times\cdots\times\Psi=\Big\{\psi_{\bn}(x)=\psi_{n_1}(x_1)\cdots\psi_{n_d}(x_d)\mid \bn=(n_1,\dots,n_d)\in\SN^d\Big\}.
\]
Clearly, if $\Psi$ is unconditional, so is $\Psi^d$.
Democracy, however, does not transfer, but one has the following result from \cite{KPT06} (see also \cite[Ch. 8]{Tem18}). 
\begin{proposition}\label{gbT2} Let $1<q<\infty$ and let $\Psi$ be a greedy basis for $L^q[0,1]$. Then for any $\La$ with $|\La|=m\geq 2$, we have:
 for
$2\le q<\infty$
$$
C^1_{q,d} \,m^{1/q}\min_{\bn\in \La}|c_{\bn}|\le\|\sum_{\bn\in \La} c_{\bn}\psi_{\bn}\|_q \le C^2_{q,d}\,(\log m)^{h(q,d)}\,m^{1/q}
\,\max_{\bn\in \La}|c_{\bn} |,  
$$
and for $1<q\le 2$
$$
C^3_{q,d}m^{1/q}(\log m)^{-h(q,d)}\min_{\bn\in \La}|c_{\bn} |\le \|\sum_{\bn\in \La} c_{\bn}\psi_{\bn}\|_q \le C^4_{q,d}m^{1/q} \max_{\bn\in \La}|c_{\bn} |
$$
where $h(q,d) := (d-1)|1/2-1/q|$.
\end{proposition}

Note that $h(p,d) = h(p',d)$. We now derive the property $\At$ for $\Psi^d$.

\begin{lemma}
Consider the system $\Psi^d$ in $L^p([0,1]^d)$ defined above, where $\Psi$ is a univariate greedy basis in $L^p[0,1]$.
Then, for each $N\geq 1$, the set $\Sigma_N(\Psi^d)$ satisfies property $\At(r,V)$ with $r=1/p'$ and
\[
V= C(p,d)\,(1+\log N)^{(d-1)(\frac1p-\frac12)_+}\,.
\]
\end{lemma}
\begin{proof}
By Lemma \ref{L_A3}, it suffices to compute 
\[
\sup_{|\La|\leq N}\big\|\sum_{\bn\in \La}\psi^*_\bn\Big\|_{L^{p'}},
\]
where $\psi^*_\bn=\psi^*_{n_1}\otimes\ldots\otimes\psi^*_{n_d}$, $\bn
\in\SN^d$, are the elements of the dual system $(\Psi^d)^*$. Since $\Psi^*$ is a univariate greedy basis (by \cite[Thm 5.1]{DKKT}),
we can apply Proposition \ref{gbT2} with $q=p'$, and obtain
\[
\sup_{|\La|\leq N}\big\|\sum_{\bn\in \La}\psi^*_\bn\Big\|_{L^{p'}}\,\leq\, C_{p,d}\,(1+\log N)^{(d-1)(\frac1p-\frac12)_+}\,N^{1/p'}.
\]
\end{proof}
 \begin{theorem}\label{gbP3} Let $1<p<\infty$.
Consider the system $\Psi^d$ in $L^p([0,1]^d)$ defined above, where $\Psi$ is a univariate greedy basis in $L^p[0,1]$.
Then, the WCGA in $L^p([0,1]^d)$ applied to $\Psi^d$ satisfies 
\begin{equation}\label{gb5}
\|f-\G_{\oldphi(m)}(f)\|_p \le C\sigma_m(f)_p,
\end{equation}
with $\oldphi(m)=C(t,p,d,\Psi)\,(1+\log m)^{(d-1)p'(\frac1p-\frac12)_+}\,m^{\al(p)}$, and with $\al(p)$ as in \eqref{alp}.
\end{theorem}
\begin{proof} 
This is again a direct application of Theorem \ref{th_tem14} and the previous lemma.
\end{proof}


\subsection{Nikol'skii $\ell^1\SX$ property}

When $\Psi=\{\psi_k\}_{k=1}^\infty$ is a basis in $L^p := L^p([0,1)^d)$, $1<p<\infty$,
then the $\At$ property essentially amounts to compute the following operator norms
$$
\|S_\La \|_{p,A} := \|S_\La \|_{L^p\to A}:= \sup_{f\neq 0} \|S_\La (f)\|_A/\|f\|_p\,,
$$
where we denote, for each finite set $\La\subset \N$,
$$
S_\La(f) := S_\La (f,\Psi) := \sum_{k\in \La } c_k(f)\psi_k,\qquad \|S_\La (f)\|_A:= \sum_{k\in \La } |c_k(f)|,
$$
and as usual $c_k=\lan f{\psi^*_k}$. Then $\At(r,V)$ holds if and only if
\[
\|S_\La \|_{p,A} \leq V\,|\La|^r.
\]
The notation refers to $A=\{f\in L^1\mid \sum_{k\in \La } |c_k(f)|<\infty\}$,
which for the trigonometric system is the usual Wiener algebra.

\subsection{General lower bounds}

We say that a Banach space $\SX$ has type $t$
if there exists a universal constant $C$ such that for $f_k\in \SX$ and all $n\geq 1$,
\Be\label{ip2}
\left({\Ave}_{\varepsilon_k=\pm 1}\|\sum_{k=1}^n\varepsilon_k  f_k\|^t\right)^{1/t}
\leq C\left(\sum_{k=1}^n \| f_k\|^t\right)^{1/t}.
\Ee
So, given $\La\subset \N$, if we set $f_k=\psi_k$, $k\in \La $, then there exists a function $f=\sum_{k\in \La } \varepsilon_k \psi_k$ with 
$\|f\|_p \le C|\La| ^{1/t}$. Thus,
\[
\|S_\La  f\|_A=|\La| =|\La| ^{1/t'}|\La| ^{1/t}\geq C^{-1}\,|\La| ^{1/t'}\,\|f\|_p.
\]
For $\SX=L^p$ it is well-known that $t=\min\{p,2\}$ if $1<p<\infty$.
Therefore, the following lower bounds always hold
\Be
\|S_\La \|_{p,A}\gtrsim |\La| ^{1/{p'}} \mbox{ if $1<p\leq 2$},\quad\mbox{and}\quad
 \|S_\La \|_{p,A}\gtrsim |\La| ^{1/{2}} \mbox{ if $p>2$}.
\label{lower_pA}
\Ee
As we shall see below, there are bases $\Psi$ for which the symbols ``$\gtrsim$'' in \eqref{lower_pA} can be replaced by ``$\approx$'' for all $\La$.
Notice, from Theorem \ref{th_tem14}, that the WCGA will be most effective in those situations, since the exponents will satisfy $rq'=1$, and therefore we have
$\oldphi(N)=c(t,p)\,(\log U)\, N$. 

\subsection{Schauder bases}

Let $\Psi$ be a (normalized) Schauder basis in $L^p$, $1<p<\infty$. Then, the following general upper bound holds.

\begin{proposition}
There exists $\e=\e(p,\Psi)>0$ and $c>0$ such that
\[
\|S_\La \|_{p,A}\leq c\,|\La| ^{1-\e},\quad \forall\;\La\subset \N.
\]
\end{proposition}
\begin{proof}
Indeed, by a well-known theorem of Gurari-Gurari \cite{gurari}, there exists $s<\infty$ and $c>0$ such that 
\[
\|\sum_{n=1}^\infty a_n\psi_n\|_p\geq \frac1c\,\Big(\sum_{n=1}^\infty|a_n|^s\Big)^{1/s},
\]
for all finite sequences of scalars $a_n$. So if $f=\sum_n a_n\psi_n$ we have
\[
\|S_\La f\|_A=\sum_{n\in \La }|a_n|\leq |\La| ^{1/s'}\,\Big(\sum_{n=1}^\infty|a_n|^s\Big)^{1/s}\leq c|\La| ^{1-\frac1s}\,\|f\|_p.
\]
Thus, the result holds with $\e=1/s$.
\end{proof}

\BR
When no further assumption is made on the basis, the bound in the proposition cannot be improved, even if $p=2$.
Indeed, for every $\al<1$, consider the basis $\Psi$ in $L^2$ constructed in \cite[Proposition 3.10]{GW14}.
This basis satisfies the following property: if $f=\sum_{n=1}^{2N}\psi_n$ and $\La=\{1,3,\ldots, 2N-1\}$, then
\[
\|S_\La f\|_2\geq c N^\al\,\|f\|_2.
\]
Since $\|S_\La f\|_2\leq \|S_\La f\|_A$ (by the triangular inequality), we conclude that
\[
\|S_\La \|_{2,A}\geq c \,|\La| ^\al, \quad \forall\; \La\subset  2\N-1.
\]
\ER

\subsection{Uniformly bounded orthogonal bases} 
Suppose that $\Psi=\{\psi_k\}_{k\geq1}$ is an orthonormal system in $L^2([0,1]^d)$ such that $c_\Psi:=\sup_{k\geq1}\|\psi_k\|_{L^\infty}<\infty$.
These systems satisfy $c_\Psi^{-1}\leq \|\psi_k\|_p\leq c_\Psi$, for all $k\geq1$ and all $1\leq p\leq \infty$.

\subsubsection{Case $2\le p<\infty$} 
\begin{proposition}\label{ipP1} If $\Psi$ is a uniformly bounded orthogonal basis of $L^p[0,1]$, $2\le p<\infty$, then
$$
\|S_\La (\cdot,\Psi)\|_{L^p\to A} \asymp |\La| ^{1/2},\quad \forall\;\La .
$$
\end{proposition}
\begin{proof}
The lower bound was shown in \eqref{lower_pA}. The upper bound follows from an elementary argument, as in \cite[Example 1]{Tem14}.
\end{proof}

\subsubsection{Case $1< p\le 2$.}
In this case the following holds
\Be
c_1\,|\La| ^{1/2}\leq \|S_\La \|_{p,A} \leq\,c_2\, |\La| ^{1/p},\quad \forall\;\La ,
\label{ip3}
\Ee
and some constants $c_i(p,\Psi)>0$.
The lower bound is due to $\|S_\La \|_{p,A}\geq \|S_\La \|_{2,A}=|\La|^{1/2}$. The upper bound follows
from the argument in  \cite[Example 1q]{Tem14}; see also \cite[Ch. 8]{Tem18}.
The example of the univariate trigonometric system shows that the bounds in (\ref{ip3}) cannot be improved
for this class of bases. 

\BR
It is known that uniformly bounded orthogonal bases cannot be unconditional in $L^p$, for any $p\not=2$; see \cite{gap58}.
Thus, when $p>2$ we cannot hope to obtain $\oldphi(N)=O(N)$ from Theorem \ref{th_tem14} using these examples. 
However, it is possible to construct such bases with $\oldphi(N)=O(N\log\log N)$; see Proposition \ref{ipP3} below.
\ER
 
 \subsection{Unconditional bases} 
Recall that a Banach space $\SX$ has cotype $q$ if there exists $c>0$ such that
\Be
\Ave_{\e_k=\pm1}\,\Big[\|\sum_{n=1}^N \e_nf_n\|\Big]\,\geq\, c\,\Big[\sum_{n=1}^N\|f_n\|^q\Big]^\frac1q.
\label{Ecotype}
\Ee
If $\Psi$ is an unconditional basis, and we let $f_n=a_n\psi_n$, then we have
\Be
\|\sum_{n=1}^Na_n\psi_n\|\geq c_2\,\Big[\sum_{n=1}^N|a_n|^q\Big]^\frac1q,
\label{anlq}
\Ee
with a constant $c_2=c_2(\Psi,c)>0$. Thus, if $f=\sum_n a_n \psi_n$, for a finite sequence $a_n$, we have 
\Be
\|S_\La f\|_A=\sum_{n\in \La }|a_n|\leq |\La| ^{1/q'}\,\Big[\sum_{n}|a_n|^q\Big]^\frac1q\leq c_2^{-1}\,|\La| ^{1/q'}\,\|f\|.
\label{upp_uncond}
\Ee
Now, specializing to the case when $\SX=L^p$, we know that $q=\max\{2,p\}$.
Therefore, from \eqref{upp_uncond} and \eqref{lower_pA} we obtain the following

\

\sline {\bf Case $1< p\leq2 $.} 
\Be
c_1\,|\La| ^{1/p'}\leq \|S_\La \|_{p,A} \leq\,c_2\, |\La| ^{1/2},\quad \forall\;\La ,
\label{uncond_pl2}
\Ee

\sline {\bf Case $2\le p<\infty$.} 
\Be
c_1\,|\La| ^{1/2}\leq \|S_\La \|_{p,A} \leq\,c_2\, |\La| ^{1/p'},\quad \forall\;\La ,
\label{uncond_pg2}
\Ee

These inequalities cannot be improved for the whole class of unconditional bases. Indeed, 
when $\Psi$ is the univariate Haar basis we have, for all $1<p<\infty$,
\Be
\|S_\La \|_{p,A} \asymp |\La| ^{1/p'},\quad \forall\; \La.
\label{SQhaar}
\Ee
On the other hand, it is known that $L^p\approx L^p\oplus \ell^2$, when $1<p<\infty$. 
Consider the basis $\{\phi_n\}_{n=1}^\infty$ in $L^p\oplus\ell^2$,
 where $\{\phi_{2n-1}\}$ is the Haar system  in $L^p$, and $\{\phi_{2n}\}$ is the canonical basis in $\ell^2$.
Then, the above isomorphim produces a (seminormalized)  unconditional basis $\Psi$ in $L^p$, with the property
\Be
\|S_\La \|_{p,A}\,\approx\,|\La| ^{1/2}, \quad \forall\;\La\subset 2\N.
\label{Lpl2}
\Ee

\subsection{Quasi-greedy bases} 
If $\Psi$ is a (normalized) quasi-greedy basis in a Banach space $\SX$ of cotype $q$, then letting $f_n=\psi_n$ in \eqref{Ecotype}
we obtain
\[
\|\sum_{n\in \La }\psi_n\|\geq c_1\,|\La| ^{1/q},
\]
for some $c_1(\Psi)>0$. This is weaker than \eqref{anlq}, but using \cite[Lemma 4.1]{GrNi01}
we have
\[
\|(a_n)\|_{\ell^{q,\infty}}\leq C\,\|\sum_n a_n\psi_n\|. 
\]
If $f= \sum_n a_n\psi_n$, for a finite sequence $(a_n)$, and $(a^*_j)$ denotes its decreasing rearrangement, then  
\Bea
\|S_\La f\|_A  =  \sum_{n\in \La }|a_n| & \leq &  \sum_{j=1}^{|\La| }a^*_j\,\leq\, \|(a_n)\|_{\ell^{q,\infty}}\,\sum_{j=1}^{|\La| }{j^{-\frac1q}}\nonumber\\
& \leq &  C'\,|\La| ^{1/q'}\,\|f\|.
\label{upp_qreedy}
\Eea
When $\SX=L^p$, using the corresponding value for the cotype $q$, one obtains again the bounds in \eqref{uncond_pl2} and \eqref{uncond_pg2}.

From the examples in \eqref{SQhaar} and \eqref{Lpl2} it is clear that these bounds cannot be improved in the class of all quasi-greedy bases.
Notice further that within this class one can replace the example in \eqref{Lpl2} by the following construction.

\begin{proposition}\label{ipP3} There exists a uniformly bounded orthonormal system $\Psi=\{\psi_k\}_{k=1}^\infty$,
consisting of trigonometric polynomials,
which is a quasi-greedy basis in $L^p[0,1]$, for all $1<p<\infty$. Moreover,
 $\Psi$ is democratic with $\|\sum_{k\in\La}\psi_k\|_p\asymp |\La|^{1/2}$, for all finite $\La\subset\SN$.
\end{proposition}
For details on this construction, we refer to \cite{N} and \cite{DS-BT} (see also \cite[Ch 3]{VTsparseb} and \cite[Ch 3]{Tem18}).
The example $\Psi$ in Proposition \ref{ipP3} produces the bounds
\Be
\|S_\La \|_{p,A}\,\approx\,|\La| ^{1/2}, \quad \forall\;\La\subset \N.
\label{ol}
\Ee
Also, quasi-greedness implies that  $\Sigma_N$ has property $\Au$ with $U\lesssim \log N$; 
see \cite[Lemma 8.2]{DKK}. 
In particular, when $p>2$, the Lebesgue inequality for the WCGA which one obtains from Theorem \ref{th_tem14} holds with
$\oldphi(N)=O(N\,\log\log N)$; see \cite{Tem14} or \cite[p. 445]{Tem18}. So far, we do not know any example of a basis (or even a dictionary) in $L^p$, $p>2$, with $\oldphi(N)=O(N)$.  
 
\subsection{Almost greedy bases}

This class is a subset of the previous case, and therefore $\|S_\La \|_{p,A}$ satisfies the same bounds 
\eqref{uncond_pl2} and \eqref{uncond_pg2}. The examples in \eqref{SQhaar} and \eqref{ol} show that the assumption that $\Psi$ is additionally democratic 
does not imply any improvement in those bounds.

 \subsection{Greedy bases} This class is also a subset of the previous three cases, but this time we obtain
the following improvement. The proof follows easily from the same arguments we already gave in \eqref{gb2}.
 
 \begin{proposition}\label{ipP2} Let $\Psi$ be a greedy basis of $L^p$, $1< p<\infty$. Then 
\Be\label{ip7}
\|S_\La \|_{p,A} \asymp |\La| ^{1/p'},\quad \forall\; \La\subset  \N.
\Ee
\end{proposition}
 
 %
  %

\section{Appendix 1}
\setcounter{equation}{0}\setcounter{footnote}{0}
\setcounter{figure}{0}

In this section we prove the following result, which was asserted in \eqref{hr_fpq}.

\begin{proposition}\label{p_app1}
Let $\SX=\fpq(\sRd)$, $1<p,q<\infty$, with norm defined as in \eqref{norm_fpq}.
Then
\Be
\sup_{|A|\leq N}\big\|\sum_{n\in A}\be_I\big\|_{\fpq(\sRd)}\,\leq\, c\, N^\frac1p\,(1+\log N)^{(d-1)(\frac1q-\frac1p)_+}\,.
\label{app1}
\Ee
\end{proposition}

For $q=2$ this was proved in \cite{Wo00, KPT06}. Here we adapt the arguments to the case of a general $q$.
We split the proof into several lemmas, which have an independent interest.
Recall that, for a sequence $\bx=(x_\la)_{\la\in\La}$, over a set of indices $\La$, we define its support as $\supp\bx=\{\la\in\La\mid x_\la\not=0\}$.

\begin{lemma}
Let $p\geq q$. If $\bx_0,\bx_1,\bx_2,\ldots$ have disjoint supports, then 
\Be
\Ts \Big\|\sum_{n\geq0}\bx_n\Big\|_{\fpq}\,\leq\,\Big(\sum_{n\geq0}\big\|\bx_n\big\|^q_{\fpq}\Big)^{1/q}.
\label{fpqq}
\Ee
\end{lemma}
\begin{proof}
The support condition implies that
\[
\Ts S_q(\sum_{n\geq0}\bx_n)^q=\sum_{n\geq 0}S_q(\bx_n)^q.
\]
Then, the definition of norm in \eqref{norm_fpq} and Minkowski's inequality (since $p/q\geq 1$) give
\Beas
\Ts \Big\|\sum_{n\geq0}\bx_n\Big\|_{\fpq} & = & \Ts\Big\|S_q(\sum_{n\geq0}\bx_n)\Big\|_{L^p}\,=\, \Big\|\sum_{n\geq 0}S_q(\bx_n)^q\Big\|_{L^{p/q}}^{1/q}\\
& \leq & \Ts\Big(\sum_{n\geq0}\big\|S_q(\bx_n)^q\big\|_{L^{p/q}}\Big)^{1/q}\,=\, \Big(\sum_{n\geq0}\big\|\bx_n\big\|^q_{\fpq}\Big)^{1/q}.
\Eeas
\end{proof}

In the next two lemmas we shall assume that $d=1$, so that $\sR_1=\sD$.
We denote by $\ell^{p,q}=\ell^{p,q}(\sD)$ the discrete Lorentz space indexed by $\sD$. 
We shall use the following (equivalent) quasi-norm:
if $\bx=(x_I)_{I\in\sD}$ and if $(x^*_j=|x_{I_j}|)_{j\geq 1}$ is its decreasing rearrangement, 
then 
\[
\big\|\bx\big\|_{\ell^{p,q}}\,:=\,\Big(\sum_{j\geq0} |2^{j/p}x^*_{2^j}|^q\Big)^{1/q}.
\]

\begin{lemma}
Let $d=1$ and  $p\geq q$. Then, $
\ell^{p,q}(\sD)\,\hookrightarrow\, \fpq(\sD)\,\hookrightarrow\,\ell^p(\sD)$, that is
\Be
 \|\bx\|_{\ell^{p}}\,\leq\,\|\bx\|_{\fpq}\,\leq\,c\, \|\bx\|_{\ell^{p,q}}. 
\label{fpq_lpq}\Ee
\end{lemma}
\begin{proof}
The left inequality is trivial, since $q\leq p$ implies $\fpq\hookrightarrow\mathfrak{f}_{p,p}=\ell^p(\sD)$.
We prove the right inequality. Let $\bx=(x_I)_{I\in\sD}$ and let $|x_{I_1}|\geq |x_{I_2}|\geq \ldots $ be its decreasing rearrangement.
Define
\[
\bx_n=\sum_{2^{n}\leq j<2^{n+1}} x_{I_j}\be_{I_j},
\]
so that $\bx=\sum_{n\geq0}\bx_n$ and the vectors $\bx_n$ have disjoint supports.
Then, \eqref{fpqq} holds. Moreover, for each $n\geq0$
\[
\|\bx_n\|_{\fpq}\,\leq \, x^*_{2^n}\,\Ts\big\|\sum_{2^n\leq j<2^{n+1}}\be_{I_j}\big\|_{\fpq}
\,\leq\, c_{p,q}\, x^*_{2^n}\,2^{n/p},
\]
where the last inequality is due to the $p$-democracy of the canonical basis in $\fpq(\sD)$ when $d=1$; see \cite[Prop. 3.2]{GH04}.
Thus,
\[
\|\bx\|_{\fpq}\,\leq\,\Big(\sum_{n\geq0}\|\bx_n\|^q_{\fpq}\Big)^{1/q}\,
\lesssim\, \Big(\sum_{n\geq0} |2^{n/p}x^*_{2^n}|^q\Big)^{1/q}\,=\,\|\bx\|_{\ell^{p,q}}.
\]
\end{proof}

\begin{lemma}\label{l_d1log}
Let $d=1$ and  $1<q<p<\infty$. Then, for all $A\subset\sD$ with  $|A|\leq N$, it holds 
\[
 \Big\|\sum_{I\in A} x_I\be_I\Big\|_{\fpq(\sD)}\,\leq\,c'\,\big[1+\log N\big]^{\frac1q-\frac1p}\,\Big(\sum_{I\in A}|x_I|^p\Big)^{1/p}. 
\]
\end{lemma}
\begin{proof}
By \eqref{fpq_lpq} and H\"older's inequality we have
\Be
\|\bx\|_{\fpq}\lesssim\|\bx\|_{\ell^{p,q}}\,\leq\,\|\bx\|^{1-\theta}_{\ell^{p,1}}\,\|\bx\|^\theta_{\ell^{p,p}},
\label{f_aux1}
\Ee
with $\theta=(1-\frac1q)/(1-\frac1p)$.
Another use of H\"older's inequality gives
\[
\|\bx\|_{\ell^{p,1}}\,=\,\sum_{j=0}^{\lfloor \log_2|A|\rfloor }2^{j/p}x^*_{2^j}\,\leq\,(1+\log_2|A|)^{1/p'}\,\|\bx\|_{\ell^{p,p}}.
\]
Inserting this into \eqref{f_aux1} and using that $\|\bx\|_{\ell^{p,p}}\approx\|\bx\|_{\ell^p}$
the result follows.
\end{proof}

\Proofof{Proposition \ref{p_app1}}
The proof for $p\leq q$ is trivial, since $\ell^p=\mathfrak{f}_{p,p}\hookrightarrow\fpq$, and
\[
\big\|\sum_{n\in A}\be_I\big\|_{\fpq}\leq\big\|\sum_{n\in A}\be_I\big\|_{\ell^p}=|A|^{1/p}.
\]
So from now on we consider $p>q$. The result is known for $d=1$ by the $p$-democracy of $\fpq(\sD)$; see
\cite[Prop. 3.2]{GH04}. So we proceed by induction, and will prove \eqref{app1} assuming its validity with $d$ replaced by $d-1$.
Let $A\subset\sRd$ with $|A|\leq N$, and define 
\[
A_1=\{I_1\in\sD\mid \exists\;I'\in\sR_{d-1} \;{\rm s.t.}\; I_1\times I'\in A\}
\]
and for each $I_1\in A_1$,
\[
A'(I_1)=\{I'\in\sR_{d-1}\mid I_1\times I'\in A\}.
\]
Then
\Beas
\big\|\sum_{I\in A}\be_I\|_{\fpq(\sRd)}^p & = & \int_{u'}\int_{u_1}\Big|\sum_{I_1\in A_1}\big(\sum_{I'\in A'(I_1)}\bbone_{I',p}^q\big)\bbone_{I_1,p}^q \Big|^{p/q}\,du_1\,du'
\\
& \lesssim & \int_{u'}(1+\log N)^{(\frac1q-\frac1p)p}\,\sum_{I_1\in A_1}\Big|\sum_{I'\in A'(I_1)}\bbone_{I',p}^q\big|^{p/q}\,du',
\Eeas
using Lemma \ref{l_d1log} in the inner integral, since $|A_1|\leq N$.
The last displayed expression equals
\Beas
& & \hskip-3cm (1+\log N)^{(\frac1q-\frac1p)p}\,\sum_{I_1\in A_1}\int_{u'}\Big|\sum_{I'\in A'(I_1)}\bbone_{I',p}^q\big|^{p/q}\,du'\\
&  \lesssim & (1+\log N)^{(\frac1q-\frac1p)p}\,\sum_{I_1\in A_1}\, (1+\log N)^{(d-2)(\frac1q-\frac1p)p}|A'(I_1)|\\
& = & \Big[(1+\log N)^{(d-1)(\frac1q-\frac1p)}|A|^{1/p}\Big]^p,
\Eeas
using in the middle step the induction hypothesis (since $|A'(I_1)|\leq N$).
\ProofEnd

\section*{ Acknowledgments }{The authors would like to thank the Isaac Newton Institute for Mathematical Sciences, Cambridge, for support and hospitality during the program 
 \emph{Approximation, Sampling and Compression in Data Science} where some work on this paper was undertaken;
this work was supported by EPSRC grant no EP/K032208/1. 

G.G. was supported in part by grants MTM2016-76566-P, MTM2017-83262-C2-2-P and Programa Salvador de Madariaga PRX18/451 from Micinn (Spain), and grant 20906/PI/18 from Fundaci\'on S\'eneca (Regi\'on de Murcia, Spain). 
E.H. was supported by grant MTM2016-76566-P (Spain), and by the European Union's
 Horizon 2020 research and innovation programme under the Marie Sk\l odowska-Curie grant agreement No 777822.
D.K. was supported by Simons Foundation Collaborative Grant No 636954.
V.T. was supported by the Russian Federation Government Grant No. 14.W03.31.0031.}

\bibliographystyle{plain}

\end{document}